\documentclass{amsart}
\usepackage{latexsym,amsxtra,amscd,ifthen}
\usepackage{amsfonts}
\usepackage{verbatim}
\usepackage{amsmath}
\usepackage{amsthm}
\usepackage{amssymb}
\usepackage[arrow,matrix,cmtip]{xy}
\usepackage{tikz}
\usepackage{color}   

\numberwithin{equation}{section}

\theoremstyle{plain}
\newtheorem{theorem}{Theorem}[section]
\newtheorem{lemma}[theorem]{Lemma}
\newtheorem{proposition}[theorem]{Proposition}
\newtheorem{corollary}[theorem]{Corollary}
\newtheorem{conjecture}[theorem]{Conjecture}

\theoremstyle{definition}
\newtheorem{definition}[theorem]{Definition}
\newtheorem{example}[theorem]{Example}

\newtheorem{remark}[theorem]{Remark}
\newtheorem{question}[theorem]{Question}

\makeatletter              
\let\c@equation\c@theorem  
\makeatother

\DeclareMathOperator{\hdet}{hdet}

\DeclareMathOperator{\GL}{GL}

\DeclareMathOperator{\Ext}{Ext}

\DeclareMathOperator{\Mod}{Mod}

\DeclareMathOperator{\injdim}{injdim}

\DeclareMathOperator{\bfl}{\mathfrak l}

\DeclareMathOperator{\Aut}{Aut}
\newcommand{\Autz}{\Aut_{\mathbb Z}}
\newcommand{\Autw}{\Aut_{{\mathbb Z}^w}}

\DeclareMathOperator{\Hom}{Hom}
\DeclareMathOperator{\RHom}{RHom}
\DeclareMathOperator{\GrMod}{GrMod}

\newcommand{\fm}{\mathfrak{m}}
\newcommand{\fe}{\mathfrak{e}}
\newcommand{\fu}{\mathfrak{u}}

\newcommand{\mf}{\mathfrak}

\newcommand{\mb}{\mathbb}

\newcommand{\ca}{\circledast}
\newcommand{\cc}{\circledcirc}

\begin{document}

\title{Skew Calabi-Yau algebras and homological identities}

\author{Manuel Reyes}
\address{(Reyes)
Bowdoin College, Department of Mathematics, 8600 College Station,
Brunswick, ME 04011-8486}
\email{reyes@bowdoin.edu}

\author{Daniel Rogalski}
\address{(Rogalski)
UCSD, Department of Mathematics, 9500 Gilman Dr. \#0112, La Jolla,
CA 92093-0112, USA. }
\email{drogalsk@math.ucsd.edu}

\author{James J. Zhang}
\address{(Zhang)
University of Washington, Department of Mathematics, Box 354350,
Seattle, WA 98195-4350, USA. }
\email{zhang@math.washington.edu}


\begin{abstract}
A skew Calabi-Yau algebra is a generalization of a Calabi-Yau
algebra which allows for a non-trivial Nakayama automorphism. We prove
three homological identities about the Nakayama automorphism and give several applications.   The identities 
we prove show (i) how the Nakayama automorphism of a smash product algebra $A \# H$ is related to the 
Nakayama automorphisms of a graded skew Calabi-Yau algebra $A$ and a finite-dimensional Hopf algebra 
$H$ that acts on it; (ii) how the Nakayama automorphism of a graded twist of $A$ is related to the Nakayama automorphism of $A$; 
and (iii) that the Nakayama automorphism of a skew Calabi-Yau algebra $A$ has trivial homological
determinant  in case $A$ is noetherian, connected graded, and Koszul.  
\end{abstract}




\date{February 15, 2013}

\maketitle


\setcounter{section}{-1}
\section{Introduction}
\label{xxsec0}

While the Calabi-Yau property originated in geometry, it now has incarnations
in the realm of algebra that seem to be of growing importance.
Calabi-Yau triangulated categories were introduced by Kontsevich
\cite{Ko} in 1998.  See \cite{Ke} for an introductory survey about Calabi-Yau triangulated categories.   
Calabi-Yau algebras were introduced by Ginzburg \cite{Gi} in 2006  as a
noncommutative version of coordinate rings of Calabi-Yau varieties.
Since the late 1990s, the study of Calabi-Yau categories and
algebras has been related to a large number of other topics such as
quivers with superpotentials, DG algebras, cluster algebras and
categories, string theory and conformal field theory, noncommutative
crepant resolutions, and Donaldson-Thomas invariants.   Some fundamental
questions in the area were answered by Van den Bergh recently in \cite{VdB3}.

One known method for constructing a noncommutative Calabi-Yau
algebra is to form the smash product $A \# H$ of a Calabi-Yau
algebra $A$ with a Calabi-Yau Hopf algebra $H$ that acts nicely on
$A$. This phenomenon has been studied quite broadly; for instance,
see \cite{BSW}, \cite[Section 3]{Fa}, \cite[Section 3]{IR},
\cite{LM}, \cite{LiWZ}, \cite{WZhu}, and
\cite{YuZ1,YuZ2}. One of the most general results in this
direction~\cite{LiWZ} states that under technical hypotheses on
Calabi-Yau algebras $A$ and $H$, $A \# H$ is Calabi-Yau if and only
if the homological determinant of the $H$-action on $A$ is trivial.
However, the smash product $A \# H$ may be Calabi-Yau even when $A$
is not Calabi-Yau; see, for instance, Example \ref{xxex1.7} below.  In order to explain
this phenomenon, it is natural to look into a broader class of
interesting algebras, called {\it skew Calabi-Yau algebras} in this paper, which
are a generalization of Ginzburg's Calabi-Yau algebras. 
While Ginzburg originally defined the Calabi-Yau property for DG-algebras
\cite[Definition 3.2.3]{Gi}, we only consider the non-DG case in this paper.

We will employ the following notation.
Let $A$ be an algebra over a fixed commutative base field $k$.  The unmarked tensor
$\otimes$ always means $\otimes_k$.  Let $M$ be an
$A$-bimodule, and let $\mu,\nu$ be algebra automorphisms of $A$. Then
${^\mu M^\nu}$ denotes the induced $A$-bimodule such that ${^\mu M^\nu}=M$ as a $k$-space, 
and where  $$a \ca m\ca b=\mu(a)m\nu(b)$$ for all $a,b\in A$ and all $m \in  {^\mu M^\nu}(=M)$. Let $A^e$ denote
the enveloping algebra $A\otimes A^{op}$, where $A^{op}$ is the opposite
ring of $A$.  An $A$-bimodule can be identified with a left $A^e$-module naturally,
or with a right $A^e$-module since $A^e$ is isomorphic to $(A^e)^{op}$.   
We usually work with left modules, unless otherwise stated.  For example, $\Hom_A(-,-)$ and
$\Ext^d_A(-,-)$ are defined for left $A$-modules. A right $A$-module
is identified with a left $A^{op}$-module.  Let $B$-$\Mod$ be the category of left $B$-modules for a ring $B$.   

\begin{definition}
\label{xxdef0.1}
Let $A$ be an algebra over $k$.
\begin{enumerate}
\item
$A$ is called {\it skew Calabi-Yau} (or {\it skew CY}, for short) if
\begin{enumerate}
\item[(i)]
$A$ is homologically smooth, that is, $A$ has a projective resolution
in the category $A^e$-$\Mod$ that has finite length and such that each term in
the projective resolution is finitely generated, and
\item[(ii)]
there is an integer $d$ and an algebra automorphism $\mu$ of $A$ such
that
\begin{equation}
\label{E0.1.1}\tag{E0.1.1}
\Ext^i_{A^e}(A, A^e)\cong \begin{cases} 0 & i\neq d\\
{^1A^\mu} & i=d, \end{cases}
\end{equation}
as $A$-bimodules, where $1$ denotes the identity map of $A$.
\end{enumerate}
\item
If \eqref{E0.1.1} holds for some algebra automorphism $\mu$ of $A$,
then $\mu$ is called the {\it Nakayama automorphism} of $A$, and
is usually denoted by $\mu_A$.  It is not hard to see that $\mu_A$ (if it exists) is
unique up to inner automorphisms of $A$ (see Lemma~\ref{xxlem1.10} below).
\item
\cite[Definition 3.2.3]{Gi}
We call $A$ {\it Calabi-Yau} (or CY, for short) if $A$ is skew
Calabi-Yau and $\mu_A$ is inner (or equivalently, $\mu_A$ can be
chosen to be the identity map after changing the generator of the bimodule
${^1A^\mu}$).
\end{enumerate}
\end{definition}
There is some variation in the literature concerning the exact definition of (skew) CY algebras.  
Ginzburg included in his definition of a CY algebra \cite[Definition 3.2.3]{Gi} the condition that 
the $A^e$-projective resolution of $A$ is self-dual, but Van den Bergh 
has shown that this is satisfied automatically \cite[Appendix C]{VdB3}.   We are also not the 
first to study skew CY algebras; the notion has been called twisted CY in several recent papers (see the beginning 
of Section 1 for more discussion).

We are primarily interested here in the special case of graded algebras $A$.   If $A$ is ${\mathbb Z}^w$-graded, 
Definition \ref{xxdef0.1} should be made in the category of
${\mathbb Z}^w$-graded modules and \eqref{E0.1.1} should be replaced by
\begin{equation}
\label{E1.0.1}\tag{E1.0.1}
\Ext^i_{A^e}(A, A^e)\cong \begin{cases} 0 & i\neq d,\\
{^1A^\mu}(\bfl) & i=d, \end{cases}
\end{equation}
where ${\bfl}\in {\mathbb Z}^w$ and ${^1A^\mu}(\bfl)$ is the shift
of ${^1A^\mu}$ by degree $\bfl$, and here $\bfl$ is called the \emph{Artin Schelter (AS) index}.  
In this case, the Nakayama automorphism
$\mu$ is a $\mb{Z}^w$-graded algebra automorphism of $A$.

For simplicity of exposition, \emph{in the remainder of this introduction $A$ is
a noetherian, connected (that is, $A_0 = k$) $\mb{N}$-graded skew CY algebra.}  
Equivalently (as we will note in Lemma~\ref{xxlem1.2}), $A$ is a noetherian Artin-Schelter regular algebra.
In this case $\mu_A$ is unique since there are no non-trivial inner automorphisms of $A$.  
Suppose further that $A$ is a left $H$-module algebra for some Hopf algebra $H$, where 
each graded piece $A_i$ is a left $H$-module.    Let $A \# H$ be the corresponding smash product 
algebra; we review the definition in Section~\ref{xxsec2}, but we note here that it is the same as 
the skew group algebra $A \rtimes G$ in the important special case that  
$H = kG$ is the group algebra of a group $G$ acting by automorphisms on $A$.
There is also a well-established theory of homological determinant of a Hopf algebra
action \cite{JoZ, JiZ, KKZ} that determines a map $\hdet: H \to k$; this is reviewed in Section~\ref{xxsec3}.  
 
Our first result provides an identity for the Nakayama automorphism of $A \# H$ in terms of those of $A$ and
$H$, along with the homological determinant $\hdet$ of the action.
It helps to explain how the smash product $A \# H$ may become
Calabi-Yau even when $A$ is only skew Calabi-Yau.  The idea is that even if $\mu_A$ is not an inner automorphism, 
$\mu_{A \# H}$ may become inner; see also Corollary~\ref{xxcor0.6} below.
\begin{theorem} 
\label{xxthm0.2} Let $H$ be a finite dimensional Hopf algebra acting
on a noetherian connected graded skew CY algebra $A$, such that each $A_i$ is a left $H$-module and $A$ is a
left $H$-module algebra. Then
\begin{equation}
\label{HI1}\tag{HI1}
\mu_{A\# H}=\mu_A\# (\mu_H\circ \Xi^l_{\hdet}),
\end{equation}
where $\hdet$ is the homological determinant of the $H$-action on $A$, and $\Xi^l_{\hdet}$ is the corresponding 
left winding automorphism (see Section~\ref{xxsec1}).
\end{theorem}
Theorem \ref{xxthm0.2} will be proved in Section \ref{xxsec4}.
In the special case where $\mu_A=Id_A$ and $\mu_H=Id_H$,
Theorem \ref{xxthm0.2} partially recovers some results in a number of papers \cite{LM, Fa, WZhu, LiWZ}.   
A natural question is if Theorem
\ref{xxthm0.2} holds when $H$ is infinite dimensional or when
$A$ is ungraded. 
When $H$ is an involutory CY Hopf algebra and $A$ is an $N$-Koszul CY algebra,
the question was answered in \cite[Theorem 2.12]{LiWZ}, but in general the question is open.

By the term \emph{homological identity} which we use in the title of the paper, we mean 
an equation involving several invariants, at least one of which is defined homologically.
Equation \eqref{HI1} is of course an example, and we note that other homological identities containing the Nakayama automorphism with interesting applications have appeared in \cite[Theorem 0.3]{BZ} and \cite[Theorem 0.1]{CWZ}.  

Next, we prove a homological identity which shows how the Nakayama 
automorphism changes under a graded twist in the sense of \cite{Zh}.
Let $\sigma$ be a graded algebra automorphism of $A$ and let $A^{\tilde{\sigma}}$ denote the left 
graded twist of $A$ associated to the twisting system $\tilde{\sigma}:=
\{\sigma^n\mid n\in {\mathbb Z}\}$.   Recall that $A^{\tilde{\sigma}}$ is an algebra with the 
same underlying graded vector space as $A$, but with new product $a \ast b = \sigma^{\mid \mid b \mid \mid}(a) b$ for 
homogeneous elements $a,b$, where $\mid \mid b \mid \mid$ indicates the degree of the element $b$.   Then it is easy to check 
using the properties of graded twists that $A$ is skew CY if
and only if $A^{\tilde{\sigma}}$ is (see Theorem~\ref{xxthm5.4} below).
For a nonzero scalar $c$, define a graded algebra automorphism $\xi_c$ of $A$ by $\xi_c(a)=
c^{\mid \mid a\mid \mid} a$ for all homogeneous elements $a\in A$.

\begin{theorem}
\label{xxthm0.3}
Let $A$ be a noetherian connected graded skew CY algebra, and let $\bfl$ be the AS index of $A$. Then
\begin{equation}\label{HI2}\tag{HI2}
\mu_{A^{\tilde{\sigma}}}=\mu_A\circ {\sigma^{\bfl}}\circ
\xi_{\hdet (\sigma)}^{-1}.
\end{equation}
\end{theorem}
\noindent Note here that $\hdet$ is defined using the natural action of the group of graded automorphisms $H = \Autz(A)$ on $A$.
Theorem \ref{xxthm0.3} will be proved in Section \ref{xxsec5}.   The result has many applications; for example, 
if one wants to prove a result about the Nakayama automorphism, one can often perform a graded twist to reduce 
to the case of an algebra with a simpler Nakayama automorphism (for example, see Lemma~\ref{xxlem6.2} below).   Similarly, in some cases one can twist a skew CY algebra to obtain a CY one, and so this gives a method of producing more CY algebras; 
see Section \ref{xxsec7} for examples.

Another goal of this paper is to demonstrate a strong connection between
the Nakayama automorphism $\mu_A$ and the homological determinant $\hdet$.
This is already evidenced by identities \eqref{HI1}, \eqref{HI2}
and \cite[Theorem 0.1]{CWZ}, and it is reinforced by the next result.   
\begin{theorem}
\label{xxthm0.4}  Let $A$ be a noetherian connected graded Koszul skew CY algebra.  Then  
\begin{equation}
\label{HI3}\tag{HI3}
\hdet(\mu_A)=1.
\end{equation}
\end{theorem}
\noindent In fact, in all examples of graded skew CY algebras we have observed, $\hdet(\mu_A) = 1$, 
and we conjecture that this is true in much wider generality.  In particular, one may consider AS Gorenstein 
rings, which satisfy \eqref{E0.1.1} but for which the homological smoothness condition is 
replaced by the weaker assumption that $A$ has finite injective dimension.  In fact, homological identities \eqref{HI1} and \eqref{HI2} 
are true in the AS Gorenstein case;  we give the more general versions in the body of the paper.    We conjecture that \eqref{HI3} 
also holds for all connected graded AS Gorenstein rings (Conjecture~\ref{xxconj6.5}).

The above homological identities have some immediate consequences.  
The following result may be useful in further study of the homological determinant.
\begin{corollary} 
Let $A$ be a connected graded skew CY algebra, and assume that Conjecture~\ref{xxconj6.5} holds. 
\label{xxcor0.5}
For every $\varphi\in \Autz(A)$, then
$$\hdet \varphi=(\mu_{(A[t;\varphi])}(t)) \; t^{-1},$$
where $A[t;\varphi]$ is the corresponding skew polynomial ring.
\end{corollary}

The next result provides several methods for constructing
CY algebras starting with skew CY algebras. 
\begin{corollary}
\label{xxcor0.6}  Let $A$ be connected graded skew CY and let $\langle \mu_A\rangle$ be the subgroup of $\Autz(A)$ generated
by $\mu_A$.  Assume that $\hdet \mu_A = 1$.
\begin{enumerate}
\item The skew polynomial ring $A[t; \mu_A]$ is CY.
\item The skew group algebra 
$A\rtimes \langle \mu_A\rangle$ is CY.
\item
Suppose that $\mu_A$ has finite order. Then the fixed subring
$A^{\langle \mu_A\rangle}$  is
AS Gorenstein with trivial Nakayama automorphism.
\end{enumerate}
\end{corollary}
Using homological identity \eqref{HI2}, one may study the Nakayama automorphisms 
of the entire family of graded twists of a given skew CY algebra $A$.  Let $o(\mu_A)$ be 
the order of the Nakayama automorphism $\mu_A$.   In most cases, one can twist 
$A$ to produce an algebra with Nakayama automorphism of finite order.
\begin{corollary}
\label{xxcor0.7}
Assume that the base field $k$ is algebraically closed of
characteristic zero.   Let $A$ be connected graded skew CY with $\hdet \mu_A = 1$. Then there is a $\sigma\in \Autz(A)$ such that
$o(\mu_{A^{\tilde{\sigma}}})<\infty$.
\end{corollary}

The paper is organized as follows. Section~\ref{xxsec1} discusses various
sources of examples of skew CY algebras. Section~\ref{xxsec2} contains some
general results on Hopf algebra actions on algebras and bimodules, as well as
smash products. Section~\ref{xxsec3} includes a study of the local cohomology
of a smash product involving a finitely graded algebra, as well as a discussion of
the homological determinant for Hopf actions on certain generalized AS Gorenstein
algebras. Sections~\ref{xxsec4}--\ref{xxsec6} respectively contain the proofs
of the homological identities (HI1)--(HI3). Section~\ref{xxsec7} is devoted to
applications of these homological identities and includes proofs of
Corollaries~\ref{xxcor0.5}--\ref{xxcor0.7}. In addition, some questions and
conjectures are presented in Sections~\ref{xxsec4}, \ref{xxsec6}, and \ref{xxsec7}.

We plan to study more results about the order of the Nakayama automorphism in a second paper \cite{RRZ}.  
In particular, there is an example of a skew CY algebra $A$
such that $A^{\tilde{\sigma}}$ is not CY for any graded algebra automorphism $\sigma$ \cite{RRZ}.  Thus 
$$\min \{o(\mu_{A^{\tilde{\sigma}}}) \mid  \sigma\in \Autz(A)\}$$
is a non-trivial numerical invariant of $A$.  We hope this invariant might be helpful in the project of
classifying Artin-Schelter (AS) regular algebras of global dimension four. 
For example, if $A$ is AS regular of dimension 4 and generated by 2 elements in degree 1, then, up
to a graded twist, we can show that $\mu_A$ is of the form $\xi_c$ where $c^7=1$ 
\cite{RRZ}. 

To close, we note that it would be useful to develop a more general theory 
of skew CY algebras parallel to that of CY algebras.  For example, 
much work on CY algebras has focused on showing how these arise in many important cases 
from potentials.  While in this paper we do not pursue this point of view, we 
note that the theory of potentials has been generalized to the skew CY setting; for example, 
``twisted superpotentials'' are discussed in \cite[Section 2.2]{BSW}.  It would be interesting 
to work out some basic properties of such generalized potentials and relate them to our results 
in this paper.  We also mention that skew Calabi-Yau categories have been considered by
Keller, who suggests that ``conservative  functors'' may serve as a good replacement for the identity functor in the definition 
of Calabi-Yau categories.

\subsection*{Acknowledgments}
The authors thank Raf Bocklandt, Bernhard Keller, Travis Schedler, Paul Smith, 
Michael Wemyss, and Quanshui Wu for useful conversations.  We thank Bernhard Keller
for sharing his ideas on defining skew Calabi-Yau categories.  We are also grateful to
Chelsea Walton and Amnon Yekutieli for helpful comments on a draft of this paper.
Reyes was partially supported by a University of California President's
Postdoctoral Fellowship during this project.  Rogalski was partially supported by the US National Science Foundation 
grants DMS-0900981 and DMS-1201572.  Zhang was partially supported by the US National Science Foundation 
grant DMS-0855743.

\section{Examples of skew CY algebras}
\label{xxsec1}

In this section, we set the scene by pointing out a number of known examples of skew CY algebras.  
We also review some definitions and terminology which we will need in the sequel.
The skew CY property for noncommutative algebras has been studied under other names 
for many years, even before the definitions of CY algebras and categories.  
If $A$ is Frobenius (finite dimensional, but not necessarily homologically smooth), then the automorphism $\mu_A$
as defined in \eqref{E0.1.1} is the Nakayama automorphism of $A$ in
the classical sense.   Similar ideas were considered in \cite{VdB1,
YeZ1} for graded or filtered algebras when rigid dualizing complexes were studied in the late 1990s.
The skew CY property was called ``rigid Gorenstein'' in \cite[Definition
4.4]{BZ}, was called ``twisted Calabi-Yau'' in \cite{BSW} (the word
``twisted'' was also used in the work \cite{DV}), and mentioned in
talks by several other people with possibly other names during the
last few years.   As mentioned in the introduction, we prefer the term 
``skew  Calabi-Yau'', since we will use the word  ``twisted" already in our 
study of graded twists of algebras.

The first major examples of skew CY algebras are the Artin-Schelter
regular algebras. The following definition was introduced by Artin
and Schelter in \cite{ASc}.
\begin{definition}
\label{xxdef1.1}
A connected graded algebra $A$ is called {\it Artin-Schelter Gorenstein}
(or {\it AS Gorenstein}, for short) if the following conditions hold:
\begin{enumerate}
\item
$A$ has finite injective dimension $d<\infty$ on both sides,
\item
$\Ext^i_A(k, A) = \Ext^i_{A^{op}} (k, A) = 0$ for all $i \neq d$ where
$k = A/A_{\geq 1}$, and
\item
$\Ext^d_A(k, A)\cong k(\bfl)$ and $\Ext^d_{A^{op}} (k, A) \cong k(\bfl)$ for
some integer $\bfl$.
\end{enumerate}
The integer $\bfl$ is called the {\it AS index}. If moreover
\begin{enumerate}
\item[(d)]
$A$ has (graded) finite global dimension $d$, 
\end{enumerate}
then A is called
{\it Artin-Schelter regular} (or AS regular, for short).
\end{definition}

In the original definition of Artin and Schelter, $A$ is required to
have finite Gelfand-Kirillov dimension, but this condition is sometimes omitted.  
It is useful to not require it here, since CY algebras of infinite GK-dimension 
are also of interest to those working on the subject.
It is known that if $A$ is AS regular, then the AS index of $A$ is positive \cite[Proposition 3.1]{SteZh}.
If $A$ is only AS Gorenstein, the AS index of $A$ could be zero or negative.

It is important to point out that the skew CY and AS regular properties coincide for connected graded
algebras. The following lemma is well-known (at least in the case when $A$ is noetherian).
\begin{lemma}
\label{xxlem1.2} Let $A$ be a connected graded algebra. Then $A$ is
skew CY (in the graded sense) if and only if it is AS regular.
\end{lemma}
\begin{proof} If $A$ is AS regular, then $A$ is homologically smooth
by \cite[Propositions 4.4(c) and 4.5(a)]{YeZ2} and satisfies
\eqref{E1.0.1} by \cite[Proposition 4.5(b)]{YeZ2}. Thus $A$ is skew CY.

Conversely, we assume that $A$ is skew CY. By \cite[Lemma 4.3(b)]{YeZ2},
the (left and right) global dimension of $A$ is equal to the projective
dimension of $_{A^e}A$, which is finite by the homological smoothness
of $A$.   We now work in the bounded derived category of left $A^e$-modules.
Let $U=\Ext^d_{A^e}(A,A^e)={^1 A^{\mu}}(\bfl)$ as in \eqref{E1.0.1}.
Since $A$ is homologically smooth, by \cite[Theorem 1]{VdB2},
\begin{equation}
\label{E1.2.1}\tag{E1.2.1}
\RHom_{A^e}(A, N)\cong U[-d]\otimes^L_{A^e} N
\end{equation}
for any left $A^e$-module $N$, where $[n]$ means complex shift (or suspension) by degree $n$.
Let $\sigma$ be any graded automorphism of $A$ 
and let $P^\bullet$ be a graded projective resolution of ${^1 A^\sigma}$ as a
left $A^e$-module.   Let $k = A/A_{\geq 1}$; since $A$ and $k$ are naturally $(A, A)$-bimodules, 
$A \otimes k$ is an $(A^e, A^e)$-bimodule, with outer left $A^e$-action and inner right $A^e$-action.
Then as elements of the derived category, we have  
$$(A\otimes k)\otimes^L_{A^e} {^1 A^\sigma}
=(A\otimes k)\otimes_{A^e} P^{\bullet}\cong
A\otimes_A P^{\bullet}\otimes_A k \cong
P^{\bullet}\otimes_A k\cong {^1 A^\sigma}\otimes_A k\cong k.$$
Similarly, in the derived category we have 
$${^1 A^\sigma}\otimes^L_{A^e}(A\otimes k)\cong k.$$
Now, using the above equations, hom-tensor adjunction and
\eqref{E1.2.1}, we have 
$$\begin{aligned}
\RHom_A(k,A)&\cong \RHom_A((A\otimes k)\otimes^L_{A^e} A, A)
\cong \RHom_{A^e}(A, \RHom_{A}((A\otimes k),A))\\
&\cong \RHom_{A^e}(A, A\otimes k)\cong U[-d]\otimes^L_{A^e}(A\otimes k)\\
&\cong ({^1 A^\mu}(\bfl)[-d])\; \otimes^L_{A^e} (A\otimes k)\cong k(\bfl)[-d]
\end{aligned}
$$
which means that $\Ext^i_A(k,A)=\begin{cases} 0& {\text{if  }} i\neq d\\
k(\bfl)&{\text{if  }} i=d \end{cases}$. By symmetry, a similar statement
is true for $\Ext^i_{A^{op}}(k,A)$. Therefore $A$ is AS regular.
\end{proof}
\noindent

By the previous result, the theory of skew CY algebras encompasses all of the theory of AS regular algebras (even those 
of infinite GK-dimension).
In fact, it is much more general still, as it applies to many other types of algebras, including ungraded ones.  For example, one may consider Hopf algebras which are skew CY, and we discuss this case next.
We first review a few facts about Hopf algebras; some more review will be found in the next section.
We recommend \cite{Mo} as a basic reference for the theory of Hopf algebras.

Throughout this paper $H$ will stand for a Hopf algebra $(H, m, u, \Delta, \epsilon)$ over $k$, with
bijective antipode $S$.   This assumption on $S$ is automatic if $H$ is
finite dimensional over $k$, or if $H$ is a group algebra $kG$,
and these are two important cases of interest.   We use Sweedler notation, so for example we write $\Delta(h) =
\sum h_1 \otimes h_2,$ and further we often even omit the $\sum$ from expressions.  
For an algebra homomorphism $f: H\to k$, the right winding automorphism of $H$
associated to $f$ is defined to be
$$\Xi^r_f: h\to \sum h_1 f( h_2)$$
for all $h\in H$. The left winding automorphism $\Xi^l_f$ of $H$
associated to $f$ is defined similarly, and it is well-known that both
$\Xi^l_f$ and $\Xi^r_f$ are algebra automorphisms of $H$.
Recall that the $k$-linear dual $H^* = \Hom_k(H,k)$ is an algebra with 
product given by the convolution $f \star g$, where $f \star g(h) =
\sum f(h_1) g(h_2)$.
Let $f: H \to k$ be an algebra map.  Then the functions $f: H \to k$ and $f \circ S: H \to k$
are easily proved to be inverses in the algebra $H^*$, and using this, one sees that  
$\Xi^r_f$ and $\Xi^r_{f \circ S}$ are inverse automorphisms of $H$.
This also implies that $f \circ S$ and $f \circ S^2$
are also inverses in $H^*$, and thus
\begin{equation}
\label{E1.2.2}\tag{E1.2.2}
f\circ S^2=f.
\end{equation}
As another consequence, it is easy to check that any winding automorphism commutes with $S^2$.  

A standard tool in the theory of the finite-dimensional Hopf algebras $H$ is the notion of integrals \cite[Chapter 2]{Mo}.
The theory of integrals has been extended to infinite-dimensional AS-Gorenstein Hopf algebras $H$ in \cite{LuWZ2}.  The left homological integral of such an $H$ is defined to be the $1$-dimensional $H$-bimodule $\int^l = \Ext^d_H({}_Hk, {}_H H)$, where $d = \operatorname{injdim}(H)$.  Picking $0 \neq \fu \in \int^l$, then $\fu$ is an invariant for the left $H$-action, that is, $h \fu  = \epsilon(h) \fu$ for all $h \in H$, but the right action gives 
an algebra map $\eta: H \to k$ defined by $\fu \cdot h = \eta(h) \fu$ which may be nontrivial.  By an abuse of 
notation we write $\int^l = \eta$, so that one has the corresponding winding automorphisms $\Xi^r_{\int^l}$ and $\Xi^l_{\int^l}$.  One may define the right integral $\int^r = \Ext^d_H(k_H, H_H)$ analogously; the left action on it gives 
an algebra map $\int^r: H \to k$ which is known to satisfy $\int^r = \int^l \circ S$.   Thus  $\Xi^r_{\int^r}$ and $\Xi^r_{\int^l}$ are 
inverses in the group of algebra automorphisms of $H$, and similarly for the left winding automorphisms.

\begin{lemma}
\label{xxlem1.3} Let $H$ be a noetherian Hopf algebra.  Then $H$ is AS regular in the sense 
of \cite[Definition 1.2]{BZ} if and only if $H$ is skew CY.   A Nakayama automorphism of
such a Hopf algebra $H$ is given by $S^{-2}\circ \Xi^r_{\int^l}$.  (Alternatively, $S^2\circ \Xi^l_{\int^l}$
is also a Nakayama automorphism of $H$.)
\end{lemma}
\begin{proof} By \cite[Lemma 5.2(c)]{BZ}, if $H$ is AS regular, then $H$ is homologically smooth.
Then $H$ is skew CY, and each of the given formulas is a Nakayama automorphism, by \cite[Theorems 0.2 and 0.3]{BZ} and the 
discussion in \cite[Section 4]{BZ}.  Conversely, by \cite[Theorem 2.3]{LuWZ1}, a noetherian skew CY Hopf algebra is AS regular. 
\end{proof}
\noindent The same argument as in the second half of the proof of Lemma \ref{xxlem1.2} shows that a
non-noetherian skew CY Hopf algebra is AS regular.  It is plausible
that the converse is true, but this is unknown.  
Brown has conjectured that every noetherian Hopf algebra is AS
Gorenstein \cite[Question E]{Br}. If Brown's conjecture holds, then
Lemma \ref{xxlem1.3} (together with \cite[Theorem 2.3]{LuWZ1})
implies that the class of noetherian Hopf algebras with finite
global dimension is equal to the class of noetherian skew CY Hopf
algebras.  In any case, Hopf algebras clearly give another large class of 
examples of skew CY algebras.

An interesting special case of both of the classes of examples introduced 
so far is the class of enveloping algebras of graded Lie algebras.  Note that there are many 
examples of graded Lie algebras, for example the Heisenberg algebras.
\begin{example}
\label{xxex1.4}
Let ${\mathfrak g}$ be a finite dimensional positively graded Lie
algebra. Then the universal enveloping algebra $U({\mathfrak g})$ is noetherian,
connected graded, AS regular, and CY, of global dimension
$d=\dim_k {\mathfrak g}$.
\end{example}
\begin{proof} 
This is quite well-known. For instance, the CY property of $U(\mathfrak{g})$
follows from~\cite[Theorem~A]{Ye}. We give a short proof based on
Lemma~\ref{xxlem1.3}.

Write $U=U({\mathfrak g})$. Since ${\mathfrak g}$ is
positively graded, the Hopf algebra $U$ is connected graded. Since
$\dim_k {\mathfrak g}=d<\infty$, the algebra $U$ is noetherian.  It is well-known that $U$ is in 
addition an AS regular, Auslander regular and Cohen-Macaulay domain
of global dimension $d$.   See, for example, \cite[Section 2]{FV}.  
Now by Lemma \ref{xxlem1.3}, the algebra $U$ is skew CY.
Since $U$ is connected graded and the maximal graded ideal of
$U$ is $\fm:=\ker \epsilon$, we have $\Ext^d_U(k,U)\cong k=U/\fm$ as $U$-bimodules. This
implies that the left integral of $U$ is trivial and hence $\Xi^r_{\int^l}$ is the identity map. Since $U$ is
cocommutative, $S^2$ is the identity map.  By  Lemma \ref{xxlem1.3}, 
$\mu_U$ is the identity and hence $U$ is CY. 
\end{proof}

The Nakayama automorphism of a skew CY algebra is not in general easy to compute.
For a Frobenius algebra, it is the same as the classical Nakayama automorphism and is 
related to the corresponding bilinear form (see the end of Section \ref{xxsec3}).
We have seen that there is a formula for the Nakayama automorphism in the case 
of Hopf algebras (Lemma \ref{xxlem1.3}) which reduces the problem to calculating the homological integral.
If the algebra $A$ is a connected graded Koszul algebra, then the Nakayama automorphism of $A$ can be determined 
if one knows the Nakayama automorphism of the Koszul dual $A^!$, which is Frobenius, since the 
actions of the Nakayama automorphisms on the degree $1$ pieces of $A$ and $A^!$ are dual up to sign \cite[Theorem 9.2]{VdB1}.
This also works in the more general case of $N$-Koszul algebras \cite[Theorem 6.3]{BM}.

The aim of the rest of the paper is to give formulas which help to compute the 
Nakayama automorphism by showing how the Nakayama automorphism changes 
under some common constructions.  We give one simple result of this type right now.  Recall that an element $z\in A$ is
called normal if (i) there is a $\tau\in \Aut(A)$ (the group of
algebra automorphisms of $A$) such that $za=\tau(a)z$ for all $a\in
A$, and (ii) $\tau(z)=z$. It is clear that, if $z$ is a
nonzerodivisor, then (ii) follows from (i). Let $\sigma$ be in
$\Aut(A)$. A normal element $z$ is called $\sigma$-normal if it is a
$\sigma$-eigenvector, namely, $\sigma(z)=cz$ for some $c\in
k^{\times}: =k\setminus \{0\}$.

There is a strong connection between the Nakayama automorphism and the theory of dualizing complexes, which 
we need in the next proof and a few others, especially in Lemma~\ref{xxlem3.5} below.   These 
instances are few enough that we do not review the theory of dualizing complexes here; the reader can find the basic definitions and results in the papers we reference. 

\begin{lemma}
\label{xxlem1.6} Let $A$ be noetherian connected graded AS
Gorenstein algebra and let $z$ be a homogeneous $\mu_A$-normal nonzerodivisor of
positive degree. Let $\tau$ be the element in $\Autz(A)$ such that
$za=\tau(a) z$ for all $a\in A$. Then $\mu_{A/(z)}$ is equal to
$\mu_A\circ \tau$ restricted to $A/(z)$.
\end{lemma}

\begin{proof} Let $d=\injdim A$. Considering $A$ as an ungraded algebra,
${^1 A^{\nu}}[d]$ is the rigid dualizing complex over $A$, where
$\nu^{-1}$ is the Nakayama automorphism of $A$, by
\cite[Proposition 6.18(2)]{YeZ1} (see also Lemma \ref{xxlem3.5} below and
\cite[Proposition 4.5(b,c)]{YeZ2}).  
The rigid dualizing complex
over $A/(z)$ is equal to $\RHom_A(A/(z), {^1 A^{\nu}}[d])$ by
\cite[Theorem 3.2 and Proposition 3.9]{YeZ1}.  A computation shows that 
$$\Ext^1_A(A/(z),{^1 A^{\nu}})=\Ext^1_A(A/(z),A)^{\nu}=
\{^1( A/(z))^{\tau^{-1}}\}^{\nu}
=(A/(z))^{\tau^{-1}\circ\nu},$$
as left $A^e$-modules, and that $\Ext^i(A/(z), A) = 0$ for $i \neq 1$.
Thus 
$$\RHom_A(A/(z), {^1 A^{\nu}}[d])=\Ext^1_A(A/(z),{^1 A^{\nu}})[d-1]
=(A/(z))^{\tau^{-1}\circ\nu}[d-1].$$
Thus the Nakayama automorphism of $A/(z)$ is $\mu_A\circ \tau$
restricted to $A/(z)$.
\end{proof}

Recall from the introduction that when $A$ is a ${\mathbb Z}$-graded algebra, for any 
$0 \neq c \in k$ we have the graded automorphism $\xi_c: A \to A$ with $\xi_c(a) = c^{\mid\mid  a \mid\mid} a$ 
for all homogeneous elements $a$, where $\mid\mid \;\;\mid\mid$
denotes the degree of homogeneous elements.  A ${\mathbb Z}$-graded skew CY algebra with $\mu_A=\xi_c$ 
is called $c$-Nakayama.   In a sense, graded $c$-Nakayama algebras are closer to CY
algebras than more general skew CY algebras. 
We now give an explicit example which was one of the motivating examples for our work in this 
paper calculating the Nakayama automorphism of a smash product.   It is a simple example 
of a ring $A$ which is skew CY only, but which becomes CY after passing to a skew group algebra.
\begin{example}
\label{xxex1.7}
Let $A$ be the $k$-algebra generated by $x$ and $y$ of degree 1 and subject
to the relations
$$x^2y=yx^2, \quad y^2x=xy^2.$$
This is the down-up algebra $A(0,1,0)$ \cite[p. 465]{KK}, 
which is AS regular algebra of type $S_1$, with AS index 4 \cite{ASc}. Let $z=xy-yx$,
then one checks that $xz=-zx$ and $yz=-zy$. Hence $z$ is normal and $z f=\xi_{-1}(f) z$
for all $f\in A$.

Let $B=A/(z)$. Then $\mu_B=Id_B$, as $B\cong k[x,y]$ is commutative.
Since $A$ is ${\mathbb Z}^2$-graded with $\deg(x)=(1,0)$ and
$\deg(y)=(0,1)$, $\mu_A$ must be ${\mathbb Z}^2$-graded.  Hence
$\mu_A$ sends $x$ to $a x$ and $y$ to $by$ for some $a,b\in k^\times$.
Consequently, $\mu_A(z)=\lambda z$ for some $\lambda\in k^\times$.
Applying Lemma \ref{xxlem1.6} (viewing $A$ as ${\mathbb N}$-graded),
$\mu_{B}=(\xi_{-1}\circ \mu_A)\mid_{B}$. Thus $\mu_A=\xi_{-1}$ when
restricted to $A/(z)$. Since both $A$ and $A/(z)$ are generated by $x$
and $y$, it follows that $\mu_A=\xi_{-1}$, or in other words $A$ is $(-1)$-Nakayama (hence not CY).

Consider the group $G = \{1,\xi_{-1}\}$ acting on $A$ naturally.   We claim that 
for $H =k G$, the smash product $A \# H$ (or 
equivalently, the skew group algebra $A \rtimes G$) is CY.   For those 
familiar with theory of superpotentials, one way to see this is 
to show that $A \rtimes G$ is isomorphic to a factor algebra of the path algebra $kQ$, where 
$Q$ is the McKay quiver corresponding to the action of $G$ on $A_1$; see \cite[Section 3]{BSW}.
In the case at hand, $Q$ is the quiver given in \cite[Section 5.3]{Boc} with vertices $\{1, 2 \}$, two arrows $a_1, a_2$ from $1$ to $2$, and two arrows 
$a_3, a_4$ from $2$ to $1$.  The algebra $A \rtimes G$ is the factor algebra of $kQ$ given 
by the cubic relations given by taking cyclic partial derivatives of the superpotential $W = a_1 a_3 a_2 a_4 + a_3 a_1 a_4 a_2$.  This algebra is shown to be CY of dimension~3 in~\cite[Theorem~5.4]{Boc}.

We have not given full details of the argument above, because one of the goals of our paper was to seek a deeper reason 
why a skew group algebra of a non CY algebra might become CY.  This is achieved by Corollary~\ref{xxcor0.6}(b), 
which shows that $A \rtimes G$ must be CY since $G$ is the cyclic subgroup generated by the Nakayama automorphism.

%
%
%
%

The current example is also interesting in the context of our main theorem on the Nakayama 
automorphisms of graded twists.  Let $\sigma\in \Autz(A)$ be defined by $\sigma(x)=x$
and $\sigma(y)=iy$ where $i^2=-1$.  By \cite[Theorem 1.5]{KK}, $\hdet
\sigma=i^2=-1$. Since $\sigma^4=1$, by Theorem~\ref{xxthm0.3}
the graded twist $A^{\tilde{\sigma}}$ is CY. One may check that the algebra
$A^{\tilde{\sigma}}$ is generated by $x$ and $y$ subject to
the relations $x^2y=-yx^2$ and $y^2x=-xy^2$. One may also use Lemma
\ref{xxlem1.6} together with a direct computation to show that
$A^{\tilde{\sigma}}$ is CY.
\end{example}

If $A$ is a ${\mathbb Z}^w$-graded algebra, we will let $\Autw(A)$ denote the
group of all graded algebra automorphisms of $A$.  Starting from Section \ref{xxsec2}, 
we will consider non-connected graded skew CY algebras.  In this case the Nakayama automorphism may not be
unique, as we see in the following standard lemma whose proof we leave to the reader.

\begin{lemma}
\label{xxlem1.10}
Let $A$ be a ${\mathbb Z}^w$-graded algebra and $M$ be a graded $A$-bimodule.
Suppose that every homogeneous left-invertible (respectively,
right-invertible) element of $A$ is invertible. Let $\nu, \sigma$ denote
elements in $\Autw(A)$.
\begin{enumerate}
\item
$M$ is isomorphic to $^{1}A^{\nu}$, for some $\nu$, if and only if
$M_A\cong A_A$ and $_AM\cong {_AA}$.
\item
Suppose $M\cong {^1A^\nu}$. Then an element $\fe\in M$ is a generator
for $_AM$ if and only if it is a generator for $M_A$.
\item
Let $\fe$ be a generator of $^{1}A^{\nu}$. Replacing $\fe$ to $\fe f$
for some invertible element $f$ changes $\nu$ to $\nu \circ \eta_f$, where 
$\eta_f$ is the inner automorphism  $a \mapsto f a f^{-1}$ for all $a\in A$. 
\item
$^{1}A^{\nu}$ is isomorphic to $^{1}A^{\sigma}$ if and only if
$\nu \sigma^{-1}$ is an inner automorphism.
\end{enumerate}
\end{lemma}
\noindent Of course in most common situations, for example if $A$ is ${\mathbb N}$-graded and $A_0$ is finite dimensional, or if
$A$ is a domain, then every homogeneous left-invertible (respectively,
right-invertible) element is invertible. In the situation of Lemma
\ref{xxlem1.10}, if $_AM=A\fe \cong _A A$ and $M_A=\fe A \cong A_A$, then we call $\fe$ 
a {\it generator} of the $A$-bimodule $M$.

\section{Hopf actions on bimodules and smash products}
\label{xxsec2}

In this section we collect some preliminary material about (bi)-modules over smash 
products and dualization.  Some of this material is well-known, but 
we have tried to make our presentation relatively self-contained.  
A reader who is less familiar with Hopf algebras may wish to think primarily about the 
case of group algebras $kG$ on a first reading, since the results of Sections 2-4 are 
still non-trivial in that case.

We maintain the assumptions about Hopf algebras already mentioned in the previous sections.
Let $A$ be a left $H$-module algebra; by definition
\cite[Definition 4.1.1]{Mo}, this means that $A$ is a left $H$-module such that 
$$h (ab) = \sum h_1( a)\; h_2 (b)\quad
{\text{and}} \quad h (1_A) = \epsilon(h)1_A$$ for all $h\in H$,
and all $a, b \in A$. Following \cite[Definition 4.1.3]{Mo},
the left smash product algebra $A \# H$ is defined as follows. As a
$k$-space, $A \# H = A \otimes H$, and the multiplication of
$A \# H$ is given by
$$(a \# g) (b \# h) = \sum a\; g_1 (b) \# g_2 h$$
for all $g, h \in H$ and $a,b\in A$. We identify $H$ with a
subalgebra of $A\# H$ via the map $i_H: h\to 1 \# h$ for all $h\in H$,
and identify $A$ with a subalgebra of $A\# H$ via the map
$i_A: a\to a\# 1$ for all $a\in A$.
We note the following useful formula which determines how an element
of $A$ and an element of $H$ move past each other:
\begin{equation}
\label{E2.0.1}\tag{E2.0.1} a\#h=(a\#1)(1\#h)=(1\#
h_2)(S^{-1}(h_1)(a)\#1).
\end{equation}

Let $B=A\# H$. Identifying $H$ and $A$ with subalgebras of $B$, any
left $B$-module is both a left $A$ and left $H$-module.  It is easy
to check that a $k$-vector space $M$ which is both a left $A$ and left
$H$-module has an induced left $B$-module structure restricting to
the given ones on both $A$ and $H$ if and only if the condition
$$h(am)=h_1(a)h_2(m)$$
is satisfied for all $h\in H, a\in A, m\in M$.  As a special case, if $A$ is a left
$H$-module such that $h(1_A)=\epsilon(h) 1_A$, then $A$ is a left
$H$-module algebra if and only if the left $A$-module and the left
$H$-module structures on $A$ induce a left $A \# H$-module structure
on $A$.

One can also construct a right-handed version of smash product.  We say
that $A$ is a right $H$-module algebra if it is a right $H$-module
satisfying
$$(ab)^h = a^{h_2} b^{h_1} \quad {\text{ and }}\quad
(1_A)^h = \epsilon(h) 1_A$$
for all $h \in H$ and $a, b \in A$.   While one could instead define a right $H$-module algebra 
by the more natural-seeming rule $(ab)^h = a^{h_1} b^{h_2}$, the convention we have chosen 
will allow us to more easily relate left and right smash products below.  
Note also that our convention is to write all right $H$-actions in the
exponent to avoid notational confusion.  If $A$ is
a right $H$-module algebra, then the right smash product $H \# A$ is
defined to be the tensor product $H \otimes A$ with multiplication $(h \#
a)(g \#b) = hg_2  \# a^{g_1} b$.  (We use the same symbol $\#$ for
the left and the right smash products.) As on the left, a $k$-vector space 
is a right $H \# A$-module if and only if it is a right $H$
and right $A$-module satisfying $(ma)^h = m^{h_2} a^{h_1}$ for all
$m \in M, a \in A$, and $h \in H$.

Recall that in a Hopf algebra, the antipode always satisfies the formulas 
\begin{equation}
\label{E2.0.2}\tag{E2.0.2}
(S \otimes S) \circ \Delta = \tau \circ \Delta \circ S\ \ \ \text{and}\ \ \ \epsilon \circ S = \epsilon,
\end{equation}
where $\tau: H \otimes H \to H \otimes H$ is the coordinate switch map
$(g \otimes h) \mapsto (h \otimes g)$ \cite[Proposition 1.5.10]{Mo}.
Suppose that $A$ is a left $H$-module algebra.  We may make it into a
right $H$-module algebra by defining $a^h = S^{-1}(h)(a)$ for all
$a \in A, h \in H$; this is easy to check using \eqref{E2.0.2}.  We
always use this fixed convention for making a left $H$-module algebra
into a right one, since with it we have the following nice property.
\begin{lemma}
\label{xxlem2.1}
Let $A$ be a left $H$-module algebra, and make it into a right
$H$-module algebra via $a^h = S^{-1}(h)(a)$. Then there is an algebra
isomorphism $\Psi: A \# H \cong H \# A$ given by the formula $a \# h
\mapsto h_2 \# a^{h_1}$, with inverse $\Psi^{-1}$ having the formula
$g \# b \mapsto g_1(b) \# g_2$.
\end{lemma}

\begin{proof}
We have
$$\begin{aligned}
\Psi((a \# g)(b \# h)) &= \Psi(ag_1(b) \# g_2 h) =
g_3 h_2 \# (ag_1(b))^{g_2 h_1}\\
&= g_4h_3 \# (a^{g_3h_2} b^{S(g_1)g_2h_1} )
= g_2 h_3 \# a^{g_1 h_2} b^{h_1},
\end{aligned}
$$
while
\[
\Psi(a \#g) \Psi(b \# h) = (g_2 \# a^{g_1})(h_2 \# b^{h_1}) \; =
g_2 h_3 \# a^{g_1 h_2} b^{h_1}.\; \quad
\]
Thus $\Psi$ is a homomorphism of algebras.  A dual proof shows that the map
$\Psi^{-1}: g \# b \mapsto g_1(b) \# g_2$ is a homomorphism
$H \# A \to A \# H$.

Note that both the formulas for $\Psi$ and $\Psi^{-1}$ act as the identity
on the subalgebras identified with $H$ and $A$.  Thus it is obvious that
$\Psi \Psi^{-1} = \Psi^{-1} \Psi = 1$, since both compositions are
clearly trivial when restricted to these subalgebras.
\end{proof}

From now on we identify the left and right smash products $A \# H$ and $H \#A$ via the isomorphism of the
previous lemma.  In particular, note that  if $M$ is a $k$-vector space
with left $A$ and  $H$-actions, then the condition $h(am) = h_1(a) h_2(m)$
is equivalent to $M$ being either a left $A \# H$ or a left $H \# A$-module
which restricts to the given actions of $A$ and $H$.  Similarly, the
condition $(ma)^h = m^{h_2} a^{h_1}$ is equivalent to being either a
right $A \# H$ or a right $H \#A$-module.

One of the main purposes of this section is to discuss actions of Hopf algebras on bimodules, where 
clearly one would like to require the Hopf action to interact with the bimodule structure in some nice way.
The twist by $S^i$ we allow in the following definition gives a slightly more general notion than what we have seen discussed in the literature.  We will see below that the extra generality will be useful to describe the dual of a bimodule with a Hopf action  (Proposition~\ref{xxpro2.7}).
\begin{definition}
\label{xxdef2.2}
Let $A$ be a left $H$-module algebra.  If $M$ is an $A$-bimodule with a
left $H$-action satisfying
\begin{equation}
\label{E2.2.1}\tag{E2.2.1} h (amb)= h_1(a)h_2(m)S^i(h_3)(b)
\end{equation}
for all $h\in H$,  $a,b\in A$, $m\in M$ and some fixed even integer $i$,
then $M$ is called an {\it $H_{S^i}$-equivariant $A$-bimodule}. When
$i = 0$, then $M$ is simply called an {\it $H$-equivariant $A$-bimodule}.
\end{definition}

\begin{remark}
\label{xxrem2.3}
Being an $H$-equivariant $A$-bimodule is equivalent to being a left
module over a certain algebra $A^{e} \rtimes H$ introduced by Kaygun
\cite[Definition 3.1, Lemma 3.3]{Kay}.  As a vector space,
$A^{e} \rtimes H = A \otimes A \otimes H$, with product given by
the formula
\[
(a_1 \otimes a'_1 \otimes h)(a_2 \otimes a'_2 \otimes g) =
a_1 h_1(a_2) \otimes h_3(a'_2) a'_1 \otimes h_2 g,
\]
for all $a_1 \otimes a'_1 \otimes h,\; a_2 \otimes a'_2 \otimes g\in
A^{e} \rtimes H$.   More specifically, if $M$ is an $H$-equivariant bimodule, one may check that it is 
an $A^e \rtimes H$-module via $(a \otimes a' \otimes h) \cdot m = ah(m)a'$.
\end{remark}

\begin{lemma}
\label{xxlem2.4}
Let $M$ be an $A$-bimodule with a left and right $H$-action  related by
\begin{equation}
\label{E2.4.1}\tag{E2.4.1} m^h=S^{-i-1}(h) (m)
\end{equation}
for all $m \in M, h \in H$, and some even integer $i$.
Then $M$ is an $H_{S^i}$-equivariant $A$-bimodule if and only if $M$
is a left $A\# H$-module and a right $A\# H$-module under the given
actions of $A$ and $H$.
\end{lemma}

\begin{proof}
We have already seen that $M$ being a left $A \# H$-module is equivalent
to the condition $h(am) = h_1(a)h_2(m)$ for all $h\in H, a\in A, m\in M$,
and that being a right $A \# H \cong H \# A$-module is equivalent to
$(mb)^h = m^{h_2} b^{h_1}$ for all $h \in H$, $b \in A$, $m \in M$.   If
$M$ is an $H_{S^i}$-equivariant $A$-bimodule, then in particular we have
$h(am) = h_1(a) h_2(m)$ and
\begin{align*}
(mb)^h &= S^{-i-1}(h)(mb) =  S^{-i-1}(h)_1(m) S^{i} (S^{-i-1}(h)_2)(b)
\\
 &= S^{-i-1}(h_2)(m) S^{-1}(h_1)(b)  \qquad \qquad \qquad \qquad \qquad
\text{using}\ \eqref{E2.0.2} \\
 & = m^{h_2} b^{h_1},
\end{align*}
so that $M$ is a left and right $A \# H$-module.  The converse is similar.
\end{proof}

Given any $H$-equivariant $A$-bimodule $M$, one can define an
$A\# H$-bimodule $M\# H$ with left and right $A \#H$ actions given by
formulas analogous to the multiplication in $A \#H$; the proof that it is
an $A\# H$-bimodule is virtually the same as the proof that the multiplication of $A \# H$ is 
associative.  See, for example, \cite[Lemma 2.5(1)]{LiWZ}.
We will need a generalization of this construction where we allow the bimodule to be $H_{S^i}$-equivariant and 
we smash with a twist of $H$ by an automorphism.
\begin{lemma}
\label{xxlem2.5}
Let $B = A \# H$ for some left $H$-module algebra $A$. Let $M$ be an
$H_{S^i}$-equivariant $A$-bimodule, and make $M$ into a right
$H$-module as in \eqref{E2.4.1}.  Suppose that $\sigma: H \to H$
is an algebra automorphism of $H$.
\begin{enumerate}
\item
Choose a left and right generator $\fu$ of the $H$-bimodule ${}^{\sigma} H^1$, such that
$h \fu = \fu \sigma(h)$.
Let $M \# {}^{\sigma} H^1$ be the tensor product $M \otimes
{^{\sigma} H^1}$, with left $B$-action
$$(a \# g)(m \# \fu h) = ag_1(m) \# g_2 \fu h$$
and right $B$-action
$$(m \# \fu h)(b \# k) = m h_1(b) \# \fu h_2 k.$$
If $\sigma$ satisfies the property
\begin{equation}
\label{E2.5.1}\tag{E2.5.1}
\Delta \circ \sigma = (S^i \otimes \sigma) \circ \Delta,
\end{equation}
then $M \# {}^{\sigma} H^1$ is a $B$-bimodule.
\item
Choose a left and right generator $\fu$ of the $H$-bimodule ${}^{1} H^{\sigma}$, such that
$\sigma(h) \fu = \fu h$. 
Let $^1 H ^{\sigma} \# M$ be the tensor product $^1 H ^{\sigma} \otimes M$,
with  left $B \cong H \# A$-action
$$(g \# a)(h \fu \# m) = g h_2 \fu  \# a^{h_1} m$$
and right $B$-action
$$(h \fu \# m)(k \# b) = h \fu k_2 \# m^{k_1} b.$$
If $\sigma$ satisfies the property
\begin{equation}
\label{E2.5.2}\tag{E2.5.2}
\Delta \circ \sigma = (S^{-i} \otimes \sigma) \circ \Delta,
\end{equation}
then $^1 H ^{\sigma} \# M$ is a $B$-bimodule.
\item
\eqref{E2.5.1} holds for any automorphism of the form
$\sigma = S^i \circ \Xi^r_{\gamma}$, where $\gamma: H \to k$ is an
algebra map.  Similarly, \eqref{E2.5.2} holds for any automorphism of
the form $S^{-i} \circ \Xi^{r}_{\gamma}$.
\end{enumerate}
\end{lemma}

\begin{proof}
(a)  We leave the proof of this part to the reader, since it is similar to but a bit simpler than the proof of part (b).

(b)  It is straightforward to see that $^1 H ^{\sigma} \# M$ is a left and right $B$-module, and so 
we check carefully only that $^1 H ^{\sigma} \# M$ is a $B$-bimodule.
Note first that 
\[
(am)^k = S^{-i-1}(k)(am) = S^{-i-1}(k_2)(a) S^{-i-1}(k_1)(m) =
a^{S^{-i}(k_2)} m^{k_1}.\qquad \quad
\]
Now we have
$$\begin{aligned}
\; [(g\# a)(h \fu \# m)] (k \# b) &= (gh_2 \fu \# a^{h_1} m)(k \# b) =
(gh_2 \fu k_2 \# (a^{h_1}m)^{k_1}b) \qquad\\
&= gh_2\sigma(k_3) \fu \# a^{h_1 S^{-i}(k_2)} m^{k_1} b, 
\end{aligned}
$$
while
$$\begin{aligned}
\quad (g\# a)[(h \fu \# m) (k \# b)] &= (g\# a)(h \fu k_2 \# m^{k_1} b) =
g (h \sigma(k_2))_2 \fu \# a^{(h \sigma(k_2))_1} m^{k_1} b \\
&= g h_2 \sigma(k_2)_2 \fu \# a^{h_1 \sigma(k_2)_1} m^{k_1} b.
\end{aligned}
$$
Using the hypothesis that $\Delta \circ \sigma = (S^{-i} \otimes \sigma)
\circ \Delta$, applying this formula to $k_2$  we have
$\Delta(\sigma(k_2))  = S^{-i}(k_2) \otimes \sigma(k_3)$,
which implies that the two expressions above are the same.

(c)  If $\sigma = S^i \circ \Xi^r_{\gamma}$, then
$$\begin{aligned}
\qquad\quad \Delta(\sigma(h))&=\Delta(S^{i}(h_1) S^i(\gamma(h_2)))=
\Delta(S^{i}(h_1) (\gamma(h_2)))\\
&=  S^{i}(h_1)\otimes
S^{i}(h_2)\gamma(h_3) =S^{i}(h_1)\otimes \sigma(h_2) \\
& = (S^i \otimes \sigma) \circ \Delta(h)
\end{aligned}
$$
for all $h\in H$. The other calculation is analogous.
\end{proof}

The next goal is to discuss the behavior of the constructions above 
under $k$-linear duals.    First, we recall the structure of the dual $H^* = \operatorname{Hom}_k(H,k)$ of a
finite-dimensional Hopf algebra $H$. 
\begin{lemma}
\label{xxlem2.6}
Let $H$ be a finite-dimensional Hopf algebra.  Then as $H$-bimodules,
$H^* \cong {}^1H^\sigma$ where $\sigma = \mu_{H}^{-1} =
S^{2}\circ \Xi^r_{\int^r}$.
\end{lemma}
\begin{proof}
Let $0\neq \fu\in H^*$ be a left integral of the Hopf algebra
$H^*$ \cite[Definition 2.1.1]{Mo}.  Then by definition, in the algebra $H^*$
we have $g \fu = \epsilon^*(g) \fu$ for all $g \in H^*$, where here
$\epsilon^*$ is the counit of $H^*$, given by $\epsilon^*(f)=f(1_H)$.
It is easy to check that the fact that $\fu$ is a left integral is
equivalent to
\begin{equation}
\label{E2.6.1}\tag{E2.6.1}
\sum h_1 \fu(h_2)=\fu(h) 1_H
\end{equation}
for all $h\in H$ \cite[Remark 5.1.2]{DNR}.
By \cite[Theorem 2.1.3(3) and its proof]{Mo}, $\fu$ is a generator of
the left and right $H$-module $H^*$, with $H\fu \cong H$ as left $H$-modules 
and $\fu H \cong H$ as right $H$-modules.  Therefore, by Lemma~\ref{xxlem1.10} there is an
algebra automorphism $\sigma$ of $H$ and an isomorphism $\phi: H^* \to {^1H^\sigma}$
of $H$-bimodules with $\phi(h\fu)  =  h$; in other words, we have the
formula $\sigma(h) \fu= \fu h$ for $h \in H$. The left-handed version of
the computation in the proof of \cite[Lemma 1.5]{FMS} now shows that
$\sigma=S^{2}\circ \Xi^r_{\int^r} =S^2\circ \Xi^r_{\int^l\circ S}$
(the formula \eqref{E2.6.1} is needed in the computation).
  
Since $H$ is finite-dimensional, it is also Frobenius, and so
$\sigma^{-1}$ is its Nakayama automorphism $\mu_H$ (see the discussion at the 
end of Section 3).  In fact, then the formula for $\sigma$ also follows from Lemma \ref{xxlem1.3}, since 
\[
(S^2\circ \Xi^r_{\int^l\circ S})^{-1} = (\Xi^r_{\int^l \circ S})^{-1} \circ S^{-2} = \Xi^r_{\int^l} \circ S^{-2} = S^{-2} \circ \Xi^r_{\int^l}
\]
(see the discussion of winding automorphisms and integrals in Section 1).
\end{proof}

In our applications of this section below, usually $A$ will be an ${\mathbb N}$-graded algebra 
$A = \bigoplus_{i \geq 0} A_i$ which is locally finite ($\dim_k A_i < \infty$ for all $i$) 
and which is a left $H$-module algebra, where the action of $H$ respects the grading in the sense that each $A_i$ is a left $H$-submodule.  We then say that $A$ is a \emph{graded left $H$-module algebra}.  More generally, a \emph{graded $H_{S^i}$-equivariant $A$-bimodule} will be a $\mb{Z}$-graded $A$-bimodule $M$ satisfying Definition~\ref{xxdef2.2}, such that the action of $H$ respects the grading.  
In the next result, we see why we defined both a left and right-sided version of the bimodule smash construction 
in Lemma~\ref{xxlem2.5}:  taking a $k$-linear dual naturally changes from one of the versions to the other.
When we are working with $\mb{Z}^w$-graded modules $M$, unless
otherwise noted, by $M^*$ we will always mean the \emph{graded dual}, that
is $M^* = \bigoplus_{\lambda \in \mb{Z}^w} \Hom_k(M_{\lambda}, k)$, which
is naturally again a $\mb{Z}^w$-graded vector space (where elements of
$\Hom_k(M_{\lambda}, k)$ have degree $-\lambda$).  Note that when $M$ is
locally finite, the graded dual $M^*$ remains locally finite and the
usual isomorphism $M^{**} \cong M$ holds.  
\begin{proposition}
\label{xxpro2.7}
Let $A$ be a locally finite graded left $H$-module algebra for a Hopf algebra $H$, and let $B = A \# H$.
Let $M$ be a locally finite $H_{S^i}$-equivariant $A$-bimodule, which we make into a right $H$-module as well as in
\eqref{E2.4.1}. Let $M^*$ be the graded $k$-linear dual of $M$.
\begin{enumerate}
\item
$M^*$ is a graded $H_{S^{-i-2}}$-equivariant  $A$-bimodule.
\item
Let $H$ be finite-dimensional, and assume that $i = 0$.  Then
$$(M \# H)^* \cong (H^*) \# M^* \cong {}^1 H^{\sigma} \# M^*$$
as $B$-bimodules, where $\sigma = S^{2}\circ \Xi^r_{\int^r}$.  Here, $M \# H$ is the
$B$-bimodule constructed in Lemma~\ref{xxlem2.5}(a) and
${}^1 H^{\sigma} \# M^*$ is the $B$-bimodule constructed
in Lemma~\ref{xxlem2.5}(b).
\end{enumerate}
\end{proposition}

\begin{proof}
(a) Write $l_h|_M$ for the map $M \to M$ given by left multiplication by $h$.  Similarly, $r_h|_M$ means 
the right multiplication by $h$.  By assumption, the left and right $H$-actions on $M$ satisfy
$$r_h|_{M}=l_{S^{-i-1}(h)}|_{M}$$
for all $h\in H$. Note that $r_h|_{M^*}=(l_h|_{M})^*$ and similarly $(r_h|_{M})^* = l_h|_{M^*}$. Then
$$r_h|_{M^*}=(l_h|_{M})^*=(r_{S^{i+1}(h)}|_{M})^*=l_{S^{i+1}(h)}|_{M^*}$$
for all $h\in H$. Clearly $M^*$ is still an $A$-bimodule and a $B$-module
on both sides, and so it is an $H_{S^{-i-2}}$-equivariant $A$-bimodule
by Lemma~\ref{xxlem2.4}.

(b) $M^*$ is an $H_{S^{-2}}$-equivariant $A$-bimodule by part (1), and
since $\sigma =  S^{2}\circ \Xi^r_{\int^r}$ satisfies \eqref{E2.5.2}
by  Lemma~\ref{xxlem2.5}(c), $^1 H^{\sigma} \# M^*$ is a well-defined
$B$-bimodule. Since $M \# H$ is a $B$-bimodule, $(M \# H)^*$ is also a
$B$-bimodule.

Since $H$ is finite-dimensional, we can identify $(M \otimes H)^*$
with $H^* \otimes M^*$ as a $k$-vector space.  Since we are taking graded duals and $M$ is locally finite, the bilinear form $\langle \cdot , \cdot \rangle$ given by the natural evaluation map 
$(H^* \# M^*) \otimes (M \# H) \to k$ is a perfect pairing.  We can identify
$H^*$ with $^1 H^{\sigma}$ as $H$-bimodules, by Lemma~\ref{xxlem2.6}, 
where $\fu \in H^*$, the left integral of $H^*$, is a right and left generator with $\fu h = \sigma(h) \fu$ for all $h \in H$.  
Thus to complete the proof, we need only
verify that under these identifications, the left and right $B$-module
structures on  $(M \# H)^*$ and $^1 H^{\sigma} \# M^*$ agree.  For
this it is enough to prove that the left and right $A$-module and $H$-module
structures all agree.  Note that it is most natural to think of $B$ as the right smash product $H \# A$ in the 
proof below.

Considering the left $H$-structure, choose 
$h \fu \# \phi \in  {}^1 H^{\sigma} \# M^*$, 
$(v \# w) \in M \# H$, and $k \in H$.  We have
\[
\langle (k(h \fu \# \phi),  v \# w\rangle =
\langle (h \fu \# \phi),  (v\# w) k\rangle =
 \langle (h \fu \# \phi), (v\# wk) \rangle = h \fu(wk) \phi(v),
\]
while
\[
\langle ((k \# 1)(h \fu \# \phi),  v \# w\rangle =
\langle ((k h \fu \# \phi),  v \# w\rangle =
kh \fu(w) \phi(v).
\]
These are the same because $kh \fu(w) = h \fu(wk)$ by
the definition of the left $H$-action on $H^*$.

The agreement of the right $A$-structures is similarly straightforward and we leave 
the proof to the reader. 

For the right $H$-structure, we have
\begin{gather*}
\langle ((h \fu \# \phi)k,  v \# w\rangle =
 \langle ((h \fu \# \phi),  k(v \# w)\rangle
=
\langle ((h \fu \# \phi),  (k_1(v) \# k_2 w) \rangle  \\
=
h \fu(k_2w) \phi(k_1(v)),
\end{gather*}
while
\begin{gather*}
\langle ((h \fu \# \phi)(k \# 1),  v \# w\rangle =
 \langle ((h \fu k_2 \# \phi^{k_1}),  (v \# w)\rangle
= h \fu k_2(w) \phi^{k_1}(v)
 \\
= h \fu(k_2w) \phi(k_1(v)).
\end{gather*}

Since $(h \fu \otimes \phi) = (h \# 1)(\fu \otimes \phi)$, and we
have proved that the left $H$-actions agree already, it is enough to
show that the left $A$-actions agree when acting on an element of
the special form $(\fu \# \phi)$, because of the formula $(1 \# a)(h
\# 1) = (h_2 \# 1)(1 \# a^{h_1})$ which holds in the right smash product.
  In this case we have
$$\begin{aligned}
\langle  a(\fu \otimes \phi),& v\# w\rangle
=\langle \fu \otimes \phi, (v\# w)a\rangle
=\langle \fu \otimes \phi, vw_1(a)\# w_2\rangle \\
&=\phi(vw_1(a))\fu(w_2) =\phi(v \fu(w_2)w_1(a))\\
&=\phi(v \fu(w) a) \qquad\qquad\qquad {\rm{ by}}\quad \eqref{E2.6.1}\\
&=\phi (va) \fu(w),
\end{aligned}
$$
while
\[
\langle  (1 \# a)(\fu \otimes \phi), v\# w\rangle =
\langle (\fu \otimes a\phi), v \# w \rangle =
(a\phi)(v) \fu(w) = \phi(va) \fu(w).
\]
This completes the proof.
\end{proof}

The next lemma concerns automorphisms of $A\# H$ that are
determined by automorphisms of $A$ and $H$.
\begin{lemma}
\label{xxlem2.8}
Let $A$ be a left $H$-module algebra.
Suppose that $\mu\in \Aut(A)$ and $\phi\in \Aut(H)$. Define a $k$-linear map
$\mu\# \phi: A\# H\to A\# H$ by
$$(\mu\# \phi)(a\# h)=\mu(a)\# \phi(h)$$
for all $a\in A$ and all $h\in H$, and define 
$\phi\# \mu: H\# A\to H\# A$ similarly. Then the following hold.
\begin{enumerate}
\item
$\mu\# \phi$ is an algebra automorphism of $A\# H$ if and only if
$\phi\# \mu$ is an algebra automorphism of $H\# A$.
\item
$\mu\# \phi$ is an algebra automorphism of $A\# H$ if and only if
\begin{equation}
\label{E2.8.1}\tag{E2.8.1} \phi(h)_1(\mu(b))\# \phi(h)_2=\mu(h_1(b))
\# \phi(h_2)
\end{equation}
for all $b\in A$ and all $h\in H$.
\item
If $\phi=S^{2n}\circ \Xi^l_{\eta}\circ \Xi^r_{\gamma}$ for some algebra
homomorphisms $\eta, \gamma: H\to k$, then \eqref{E2.8.1} holds
if and only if
\begin{equation}
\label{E2.8.2}\tag{E2.8.2}
(\Xi^l_{\eta}\circ S^{2n})(h)(\mu(b))=\mu(\Xi^r_{\eta}(h)(b))
\end{equation}
for all $h \in H$ and $b \in A$.
\end{enumerate}
\end{lemma}

\begin{proof} (a) Recall from Lemma \ref{xxlem2.1} that the algebra isomorphisms
$\Psi: A\# H\to H\# A$ and $\Psi^{-1}: H\# A\to A\# H$ are defined
by $\Psi(a\# h)=h_2\# a^{h_1}$ and by $\Psi^{-1}(g\#b)=g_1(b)\#
g_2$.   Suppose that $\mu \# \phi$ is an automorphism of $H \#A$.  Then 
$\Psi^{-1} (\mu\# \phi) \Psi$ is an automorphism of $A \#H$, and it is easy to see 
that $\phi\#\mu=\Psi^{-1} (\mu\# \phi) \Psi$, 
since this formula holds for elements of $A$ and elements of $H$.   The converse is similar.

(b) By definition, for all $a\# h, b\# g\in A\# H$,
$$\begin{aligned}
(\mu\# \phi)(a\# h)(\mu\# \phi)(b\# g)
&=(\mu(a)\# \phi(h))(\mu(b)\# \phi(g))\\
&=\mu(a)\phi(h)_1(\mu(b))\#\phi(h)_2\phi(g)
\end{aligned}$$
and
$$\begin{aligned}
(\mu\# \phi)((a\# h)(b\# g))
&=(\mu\# \phi)(ah_1(b)\# h_2g)\\
&=\mu(ah_1(b))\# \phi(h_2g)= \mu(a)\mu(h_1(b))\# \phi(h_2)\phi(g).
\end{aligned}
$$
The assertion follows by comparing these two equations.

(c) By definition,
$$\phi(h)=\eta(h_1) S^{2n}(h_2) \gamma(h_3)$$
for all $h\in H$. Since $S^{2n}$ is a Hopf algebra automorphism of $H$,
we have
$$
\Delta(\phi(h))=\eta(h_1) S^{2n}(h_2)\otimes S^{2n}(h_3) \gamma(h_4)
$$
and
$$h_1\otimes \phi(h_2)=h_1\otimes \eta(h_2)S^{2n}(h_3)\gamma(h_4)
=\Xi^r_{\eta}(h_1)\otimes S^{2n}(h_2)\gamma(h_3).$$ Then \eqref{E2.8.1}
in this case is the condition
$$
(\Xi^l_{\eta}\circ S^{2n})(h_1)(\mu(b))\# S^{2n}(h_2)\gamma(h_3)
=\mu(\Xi^r_{\eta}(h_1)(b))\# S^{2n}(h_2)\gamma(h_3)
$$
for all $h \in H$ and $b \in A$.  Clearly then \eqref{E2.8.2} implies \eqref{E2.8.1}.  
Conversely, if \eqref{E2.8.1} holds, then applying the 
winding automorphism $\Xi^r_{\gamma \circ S}$ to both sides of the previous displayed equation and 
then applying $1 \# \epsilon$ gives \eqref{E2.8.2}.  
\end{proof}

\begin{remark}
\label{xxrem2.9}
Recall from \cite{BZ} that the Nakayama automorphism of a noetherian AS Gorenstein Hopf algebra $H$ has the form 
$\phi=S^{-2}\circ \Xi^r_{\gamma}$.   Using this automorphism of $H$ in the previous result, \eqref{E2.8.2} becomes
$$S^{-2}(h) (\mu(b))=\mu(h(b))$$
for all $b\in A$ and $h\in H$.  This formula is related to \eqref{E3.10.2} 
and Lemma \ref{xxlem5.3}(a) below.
\end{remark}

We conclude this section with a result that will allow us to 
better understand the bimodule
$(M \# H)^*$ constructed in Proposition~\ref{xxpro2.7}, in the special
case that $M^* \cong {}^{\mu} A ^1$ for an automorphism $\mu$.
\begin{lemma}
\label{xxlem2.10} Let $A$ be a left $H$-module algebra and let  $B=A\# H$.  Let 
$\mu: A \to A$ be an algebra automorphism of
$A$ and suppose that the $A$-bimodule $N = {}^{\mu} A ^1$ has a
left $H$-action making it a $H_{S^i}$-equivariant $A$-bimodule for some
even integer $i$.  Make $N$ a right $H$-module as in \eqref{E2.4.1}.  Suppose
further that $N$ has a left and right $A$-module generator $\fe$ which is
$H$-stable in the sense that $h(\fe)  \subseteq k\fe$ for all $h \in H$,
and define $\eta: H \to k$ by $h(\fe) = \eta(h) \fe$. Suppose that 
$\sigma: H \to H$ is an automorphism  satisfying \eqref{E2.5.2}, and
let $\fu$ be a left and right generator for $^1 H ^{\sigma}$, so that
$\fu h = \sigma(h) \fu$.
\begin{enumerate}
\item
The element $\fu \# \fe$ is a left and right $B$-module generator of
$^1 H ^{\sigma} \# {}^{\mu} A ^1$, where this is the $B$-bimodule constructed in 
Lemma~\ref{xxlem2.5}(b).
\item
We have $^1 H ^{\sigma} \# {}^{\mu} A ^1 \cong {}^{\rho}
B ^1$ as  $B$-bimodules, where the automorphism $\rho$ of $B$ has the
following  formula:
\[
\rho(a \# h)=  \mu(a) \# \Xi^l_{\eta} \circ \sigma^{-1}(h).
\]
\end{enumerate}
\end{lemma}

\begin{proof} 
(a) We work with the right smash product $H \# A$ in the proof, since the
formulas for the $B$-bimodule structure on $^1 H ^{\sigma} \# {}^{\mu} A ^1$ are given 
in terms of it. 

It is straightforward to check the identity
\begin{align*}
\label{E2.10.2}\tag{E2.10.2}
(\fu \# \fe)(g\# 1) &= \fu g_2 \# \fe^{g_1} = \fu g_2 \# S^{-i-1}(g_1)(\fe)
\\
&= \fu g_2 \# \eta(S^{-i-1}(g_1)) \fe = \fu g_2  \eta(S^{-i-1}(g_1)) \# \fe
\\
&= \fu \Xi^l_{\eta \circ S}(g) \# \fe \qquad\qquad \qquad \qquad \qquad
\text{using}\ \eqref{E1.2.2}.
\end{align*}
Also, $(\fu \# \fe)(1 \# a) = (\fu \# \fe a)$ and $(g \# a)(\fu \# \fe)
= (g \fu \# a \fe)$, so it follows that $\fu \#  \fe$ is a left and right
$B$-module generator.  The same formulas easily imply that no nonzero
element of $B$ kills $\fu \#  \fe$ on either side, so $^1 H ^{\sigma}
\# {}^{\mu} A ^1$ is a free $B$-module of rank $1$ on each side.  Then this bimodule 
is isomorphic to ${}^{\rho} B ^1$ for some automorphism
$\rho: B \to B$, by Lemma~\ref{xxlem1.10}.

(b)  Since $\Xi^l_{\eta \circ S} = (\Xi^l_{\eta})^{-1}$, we calculate for any $h\in H$ that 
$$\begin{aligned}
(h\# 1)(\fu \# \fe)
&= h \fu \# \fe
\\
&=\fu \sigma^{-1}(h) \# \fe  \\
&=(\fu \# \fe)(\Xi^l_{\eta}\circ \sigma^{-1}(h) \# 1)
\quad \quad {\rm{ by}}\quad \eqref{E2.10.2}.
\end{aligned}
$$
This shows that $\rho(h \#1 )=\Xi^l_{\eta}\circ \sigma^{-1}(h) \# 1$.

For any $a\in A$,
$$
(1\# a)(\fu \# \fe)= \fu \# a\fe =\fu \# \fe \mu(a)
=(\fu \#  \fe)(1 \# \mu(a))\qquad
$$
Hence $\rho(1 \# a)= 1 \# \mu(a)$.

Thus we know how $\rho$ acts on elements of $A$ and in $H$, and so $\rho = \Xi^l_{\eta} \circ \sigma^{-1} \# \mu$ 
as an automorphism of $H \# A$.  By Lemma~\ref{xxlem2.8}(a) and its proof, we see that 
as an automorphism of the left smash product $A \# H$, we have the formula $\rho = \mu \#  \Xi^l_{\eta} \circ \sigma^{-1}$ as required.
\end{proof}

\section{AS Gorenstein algebras and local cohomology}
\label{xxsec3}

In this section, we introduce the main technical tool of our approach in this paper, 
which is the local cohomology of graded algebras.  We also define a generalization of the AS Gorenstein condition to not necessarily connected graded algebras, and discuss the homological determinant in this setting.

Let $A$ be a locally finite ${\mathbb N}$-graded algebra and $\fm_A$
be the graded ideal $A_{\geq 1}$.   Let $A-\GrMod$ denote the category of $\mb{Z}$-graded left $A$-modules.  Similarly, if $A$ and $C$ are graded algebras, then $(A, C)-\GrMod$ is the 
category of $\mb{Z}$-graded $(A, C)$-bimodules.
For each $n$ and each graded left
$A$-module $M$, we define
$$\Gamma_{\fm_A}(M)= \{x\in M\mid A_{\geq n} x=0
{\text{ for some $n \geq 1$}}\}
= \lim_{n \rightarrow \infty} \Hom_A(A/A_{\geq n}, M)$$
and call this the \emph{$\fm_A$-torsion} submodule of $M$.  It is standard 
that the functor $\Gamma_{\fm_A}$ is a left exact functor from $A-\GrMod$ to itself.
Since this category has enough injectives, the right derived functors 
$R^i\Gamma_{\fm_A}$ are defined and called the local cohomology functors.  
Explicitly, one has $R^i\Gamma_{\fm_A}(M) = \lim_{n \rightarrow \infty} \Ext_A^i(A/A_{\geq n}, M)$.
See \cite{AZ} for more background. 

We also consider ${\mathbb Z}^w$-graded algebras for a positive integer $w$.
In the ${\mathbb Z}^w$-graded setting, the
degree of any homogeneous element $a$ is denoted by $\mid a\mid=
(n_1,\cdots,n_w)\in {\mathbb Z}^w$. Define a new
${\mathbb Z}$-grading by $\mid\mid a\mid\mid =\sum_{s=1}^w n_s$.
In most cases, we will assume that $A$ is noetherian and
a locally finite ${\mathbb N}$-graded algebra with respect to the
$\mid\mid\;\; \mid\mid$-grading.   But in a few occasions, we consider
more general ${\mathbb Z}^w$-gradings (for example, in 
the discussion of Frobenius algebras at the end of this section).
The local cohomology functors for  ${\mathbb Z}^w$-graded algebras $A$ are defined 
using the $\mid\mid\;\; \mid\mid$-grading and the same torsion functor $\Gamma_{\fm_A}$.  
This is an endofunctor of the category of ${\mathbb Z}^w$-graded modules and thus the local 
cohomology modules are also in this category.  There is a forgetful functor from the
category of ${\mathbb Z}^w$-graded $A$-modules to the category of
${\mathbb Z}$-graded $A$-modules. It is easy to check that the local
cohomological functors $\Gamma_{\fm_A}$ and the derived functors
$R^d\Gamma_{\fm_A}$ commute with the forgetful functor, so we
use the same notation for local cohomological functors in either
graded category.



In the next two lemmas we consider the local cohomology of (bi)-modules over a smash product.
\begin{lemma}
\label{xxlem3.1} Let $A$ be a locally finite $\mb{N}$-graded left $H$-module algebra 
for a finite-dimensional Hopf algebra $H$, and let $B = A \# H$.   
Let $i \geq 0$ be an integer and let $C$ be another graded algebra.
Then as endofunctors of
$(B,C)$-$\GrMod$,
$$R^i\Gamma_{\fm_A}(-)\cong R^i\Gamma_{\fm_B}(-).$$
\end{lemma}

\begin{proof} Note first that $B$ is a flat right $A$-module.  In fact, it is obvious 
from thinking of $B$ as a right smash product $H \# A$, as in Lemma~\ref{xxlem2.1}, that $B$ is a free right $A$-module.  Then every (graded) injective left $B$-module is 
a (graded) injective left $A$-module; see \cite[Lemma~3.5]{La}, or see \cite[Lemma 5.1]{KKZ} for a graded version.  

Given a graded $(B,C)$-bimodule $M$, consider 
a graded injective $(B,C)$-bimodule resolution $I^{\bullet}$ of $M$.  In other words, we take an  
injective resolution of $M$ in the category $(B \otimes_k C^{op})-\GrMod$,
where $B \otimes_k C^{op}$ is also a free right $B$-module and hence the left modules $I^i$ 
are injective in both $B-\GrMod$ and $A-\GrMod$.  Thus we can use $I^{\bullet}$ to calculate either functor.  It is easy to check using
the formula \eqref{E2.0.1} that for any left $B$-module $M$,
$\Gamma_{\mf{m}_A}(M)$ is a $\fm_B$-torsion $B$-submodule of $M$. It
easily follows that for any $(B,C)$-bimodule $M$, we have
$\Gamma_{\mf{m}_A}(M) = \Gamma_{\mf{m}_B}(M)$, and that these are
$(B, C)$-sub-bimodules of $M$. This implies that $\Gamma_{\mf{m}_A}
( - ) = \Gamma_{\mf{m}_B}( - )$ as endofunctors of $(B,C)$-$\GrMod$.
The assertion easily follows by applying these functors to $I^{\bullet}$ and taking
homology.
\end{proof}

\begin{lemma}
\label{xxlem3.2}  Let $A$ be a locally finite $\mb{N}$-graded left $H$-module algebra, and let $B= A \# H$.
Suppose that $M$ is a graded $H$-equivariant $A$-bimodule.
\begin{enumerate}
\item
For any integer $d \geq 0$, $R^d\Gamma_{\fm_A}(M)$ has a natural
left $H$-action and is an $H$-equivariant $A$-bimodule.
\item
Assume $H$ is finite dimensional.  Then as $B$-bimodules, 
$$R^d\Gamma_{\fm_B}(M \#H) \cong R^d\Gamma_{\fm_A}(M\# H)
\cong R^d\Gamma_{\fm_A}(M)\# H,$$
where the bimodule structure of $R^d\Gamma_{\fm_A}(M)\# H$ is as in 
 Lemma~\ref{xxlem2.5}(a).
\end{enumerate}
\end{lemma}

\begin{proof} (a)  We have seen in Remark~\ref{xxrem2.3} that to be
a  $H$-equivariant $A$-bimodule is equivalent to being a left module
over the algebra $R = A^e \rtimes H$ of Kaygun.   So $M$ is a left $R$-module, 
and clearly in fact $R$ is a graded algebra (with $H$ in degree $0$) 
and $M$ is a graded left $R$-module.    
Then $\Gamma_{\fm_A}(M)$ is a left $B$-submodule
and a right $A$-submodule of $M$, by the same proof as in
Lemma~\ref{xxlem3.1}, and it follows that $\Gamma_{\fm_A}(M)$
is an $R$-submodule of $M$.  

Take a graded injective left $R$-module
resolution  $I^{\bullet}$ of $M$.   Since we have in $R$ that  
$(1 \otimes b \otimes 1)(a \otimes 1 \otimes h) = (a \otimes b \otimes h)$,
for all $a, b \in A$ and $h \in H$, it is easy to see that $R$ is a free right $A$-module. 
Thus $I^{\bullet}$ is also a graded injective left $A$-module resolution of $M$, by 
the same argument as in  Lemma~\ref{xxlem3.1}.    Thus we may use this resolution
to calculate $R^d\Gamma_{\fm_A}(M)$.  We conclude that
$R^d\Gamma_{\fm_A}(M)$ retains a left $R$-module structure, in other words, 
it is a graded $H$-equivariant $A$-bimodule.  In fact, it is then easy to see that to obtain the 
left $R$-module structure on $R^d\Gamma_{\fm_A}(M)$, we can use any resolution 
of $M$ by a complex of graded $R$-modules each of which is graded injective as an $A$-module (in any words, the 
modules need not be injective as $R$-modules.  

(b)  The first isomorphism follows from Lemma~\ref{xxlem3.1}.  We note
that given an left $R$-module homomorphism $\phi: M \to N$, we can take the
smash product of each module with $H$ as in Lemma~\ref{xxlem2.5}(a) to
obtain a map $\phi \# 1: M \# H \to N \# H$.  Smashing with $H$ is easily 
seen to be functorial in the sense that $\phi \# 1$ is a
$B$-bimodule map.

Again choosing a graded injective left $R$-module resolution $I^{\bullet}$
of $M$, we can smash this resolution with $H$; this gives a complex 
$I^{\bullet} \# H$ of $B$-bimodules which is a resolution of $A \# H$.  Note that 
$I^{\bullet} \# H$ is a complex of graded injective left $A$-modules, since each $I^i$ is graded 
injective over $A$ as in part (a) and $I^i \# H$ is isomorphic as a left $A$-module to a finite direct sum of copies of $I^i$ (recall that $H$ is finite dimensional).   

Now it is easy to check that for any graded left $R$-module $N$,  
$\Gamma_{\fm_A}(N \# H) =
\Gamma_{\fm_A}(N) \# H$ as subspaces of $N \# H$, and by functoriality since $\Gamma_{\fm_A}(N)$ is an $R$-submodule of $N$, it follows that
$\Gamma_{\fm_A}(N) \# H$ is a $B$-sub-bimodule of $N \# H$. Finally, using the injective resolution $I^{\bullet} \# H$ to compute the derived
functors, it follows that $R^d\Gamma_{\fm_A}(M\# H) \cong
R^d\Gamma_{\fm_A}(M)\# H$ as $B$-bimodules, for each $d$.
\end{proof}

We would like to consider a generalization of AS Gorenstein algebras to the non-connected case, but where the algebra is still locally finite.  In this paper we propose the following definition.
\begin{definition}
\label{xxdef3.3}
Let $A$ be a ${\mathbb Z}^w$-graded algebra, for some $w\geq 1$, such
that it is locally finite and ${\mathbb N}$-graded with respect to the
$\mid\mid \;\;\mid\mid$-grading. We say $A$ is a {\it generalized
AS Gorenstein algebra} if
\begin{enumerate}
\item
$A$ has injective dimension $d$.
\item
$A$ is noetherian and satisfies the $\chi$ condition (see \cite[Definition 3.7]{AZ}), and the 
functor $\Gamma_{\fm_A}$ has finite cohomological dimension.
\item
There is an $A$-bimodule isomorphism $R^d\Gamma_{\fm_A}(A)^*\cong {^\mu A^1}(-{\bfl})$ (where this is the graded dual), for some $\bfl\in{\mathbb Z}^w$ (called the {\it AS index}) and for some graded algebra
automorphism $\mu$ of $A$ (called the {\it Nakayama automorphism}).  
\end{enumerate}
\end{definition}

\begin{remark}
\label{xxrem3.4} Other notions of generalized AS regular algebras were
introduced for not necessarily connected graded algebras 
by Martinez-Villa and Solberg in \cite{M-VS} and by Minamoto and Mori
in~\cite{MM}.  Even if $A$ is noetherian of finite global dimension, 
our definition above is slightly stronger than the one in \cite{M-VS}.

\end{remark}

One of our main motivations for introducing the
definition of generalized AS Gorenstein is that when $A$ is a connected graded AS Gorenstein algebra which is a graded $H$-module algebra for some finite-dimensional Hopf
algebra $H$, then $A \# H$ will be generalized AS Gorenstein (Theorem \ref{xxthm4.1}(b)).
The hypotheses in our definition of generalized AS Gorenstein are chosen to make sure that the theory of 
dualizing complexes will work as usual.  We discuss this in the following lemma, which justifies calling $\mu$ in Definition \ref{xxdef3.3}(c) a Nakayama automorphism of $A$.

\begin{lemma}
\label{xxlem3.5}
Let $A$ be generalized AS Gorenstein, and let $\mu\in \Autz(A)$ such that
$$R^d\Gamma_{\fm_A}(A)^*\cong {^\mu A ^1}(-\bfl)$$
for some $\bfl\in {\mathbb Z}$; namely, $\mu$ is a Nakayama automorphism in the sense of
Definition \ref{xxdef3.3}. Then we have
$$\Ext^i_{A^e}(A, A^e)\cong \begin{cases} 0 & i\neq d\\
{^1A^\mu}(\bfl) & i=d \end{cases};
$$
namely, $\mu$ is a Nakayama automorphism in the sense of
Definition \ref{xxdef0.1}.
\end{lemma}

\begin{proof}
Van den Bergh's paper \cite{VdB1} works with connected graded 
algebras only.  However, one can check that the results of \cite[Sections 3-8]{VdB1}
hold with no essential change for a locally finite $\mb{N}$-graded algebra.  This generalization is similar 
to the semi-local complete case which is worked out explicitly in \cite{WZ}. 

Now \cite[Theorem 6.3]{VdB1} shows that 
$R := R^d\Gamma_{\fm_A}(A)^*[d]\cong {^\mu A^1}(-{\bfl})[d]$ is a
balanced dualizing complex for $A$. 
By \cite[Proposition 8.2]{VdB1}, it is also a rigid 
dualizing complex, and this implies in particular by \cite[Proposition 8.4]{VdB1} that 
$R^{-1} = \RHom_{A^e}(A, A^e)$, where $R^{-1}$ is the inverse of $R$ under derived tensor.  Equivalently, since the inverse of the bimodule $^{\mu} A ^1$ is isomorphic to $^1 A^{\mu}$, we have $R^{-1} \cong {^1A^\mu}(\bfl)[-d]$ and thus  
$$\Ext^i_{A^e}(A, A^e)\cong \begin{cases} 0 & i\neq d\\
{^1A^\mu}(\bfl) & i=d \end{cases},
$$
as claimed.  
\end{proof}

We note in the following remark that generalized AS Gorenstein algebras satisfy generalized versions of properties (b,c) of Definition~\ref{xxdef1.1}. 
\begin{remark}
\label{xxrem3.5}
Suppose that $A$ is generalized AS Gorenstein of injective dimension $d$ and AS index $\bfl$.  Then we claim that  for a finite-dimensional graded left $A$-module $S$ concentrated in degree $0$, we have that $\Ext^i(S, A)  = 0$ if $i \neq d$, and that $\Ext^d(S, A)$ is a finite-dimensional $k$-space concentrated in graded degree $-\bfl$.

To see this, note that $A$ satisfies the hypotheses of (the locally finite version of) \cite[Theorem 5.1]{VdB1},  and so we have the 
local duality formula 
\[
R\Gamma_{\fm_A}(M)^* = R\Hom_A(M, R\Gamma_{\fm_A}(A)^*)
\]
for any graded left $A$-module $M$.
Since we have $R\Gamma_{\fm_A}(A)^* \cong {} ^{\mu} A ^1(-\bfl)[d]$ in this case, taking $M = S$ the claim follows 
as long as $R\Gamma_{\fm_A}(S)^*$ is finite-dimensional and concentrated in graded degree $0$.  But by the locally finite version of 
\cite[Lemma 4.4]{VdB1}, since $S$ is finite-dimensional we have $R\Gamma_{\fm_A}(S)^* = S^*$.  

Clearly, one also has a right module analog of the comments above.  
Note that this implies in particular that a connected graded generalized AS Gorenstein algebra is a (noetherian) AS Gorenstein algebra in the usual sense.
\end{remark}

The homological determinant has been an important tool for understanding the theory of 
connected graded AS Gorenstein algebras, especially the invariant theory of group and Hopf algebra actions.  Next, we develop a theory of homological determinant that 
will apply, in certain cases, to generalized AS Gorenstein algebras.
\begin{definition}
\label{xxdef3.6}
Let $A$ be a $\mb{Z}^w$-graded generalized AS Gorenstein algebra.   Let $A$ be a graded left $H$-module algebra for some Hopf algebra $H$.  Since $A$ is an $H$-equivariant $A$-bimodule, then $R^d\Gamma_{\fm_A}(A)^*$ is a graded 
$H_{S^{-2}}$-equivariant $A$-bimodule by Lemma~\ref{xxlem3.2} and
Proposition~\ref{xxpro2.7}.  
Suppose further that there is a nonzero element
$\fe\in R^d\Gamma_{\fm_A}(A)^*$ of degree $\bfl$ such that
\begin{enumerate}
\item[]
\begin{enumerate}
\item[(i)]  
$\fe$ is a left and right generator of the $A$-bimodule
$R^d\Gamma_{\fm_A}(A)^*$; and 
\item[(ii)]
$k\fe$ is a left $H$-submodule.
\end{enumerate}
\end{enumerate}
\begin{enumerate}
\item
The element $\fe$ is called {\it an $H$-stable generator} of
$R^d\Gamma_{\fm_A}(A)^*$.
\item
Under the condition (i) alone, there is a graded algebra automorphism
$\mu$ such that $a\fe=\fe \mu(a)$ for all $a\in A$. We call such a
$\mu$ the {\it $\fe$-Nakayama automorphism} of $A$.  Note that by Lemma~\ref{xxlem1.10},
any other Nakayama automorphism will differ by an inner automorphism.
\item
We define an algebra homomorphism $\hdet: H\to k$ by
\begin{equation}
\label{E3.7.1}\tag{E3.7.1}
\hdet(h)\; \fe = h (\fe)
\end{equation}
for all $h\in H$ and call the map $\hdet$ the
{\it $\fe$-homological determinant} of the $H$-action.
\end{enumerate}
\end{definition}
\noindent 
Note that the $\fe$-homological determinant only depends on the 
$\mid\mid \;\;\mid\mid$-grading, and is independent of possible choices of $\mb{Z}^w$-gradings 
that induce the same $\mid\mid \;\;\mid\mid$-grading, as long as the choice of $\fe$ is fixed.
  When $A$ is connected graded 
AS Gorenstein, then any bimodule generator $\fe$ of $R^d\Gamma_{\fm_A}(A)^*$ is contained in the $1$-dimensional degree-$\bfl$ piece and so is automatically $H$-stable since the $H$-action respects the grading.  The $\fe$-Nakayama automorphism is the unique choice of
Nakayama automorphism in this case, and the homological determinant
is independent of $\fe$.  On the other hand, there is no obvious reason why an arbitrary generalized AS Gorenstein algebra should have an $H$-stable generator.  In future work, we hope to generalize some of our theorems below to work without this assumption.

\begin{remark}
\label{xxrem3.7}
In \cite[Definition 3.3]{KKZ}, the homological determinant $\hdet: H \to k$ is defined, given a finite dimensional Hopf algebra $H$ and a connected graded AS Gorenstein algebra which is a graded left $H$-module algebra.  In this case, we claim that the definition of $\hdet$ we gave above coincides with the definition in \cite{KKZ}.   Both definitions depend on first putting a left
$H$-action on $R^d\Gamma_{\fm_A}(A)$.  We did this in
Lemma~\ref{xxlem3.2} by taking a graded injective resolution $I^{\bullet}$ of $A$
over the Kaygun algebra $R = A^e \rtimes H$ and noting that the terms of this resolution are graded injective over $A$ also.    In \cite{KKZ}, a $B = A \# H$-module graded injective resolution of $A$ is used.   Similarly as in the proof of Lemma~\ref{xxlem3.2}, any $B$-module resolutions in which the terms are graded injective over $A$ must induce the same $B$-module structures on the local cohomology, so we get the same $B$-action on $R^d\Gamma_{\fm_A}(A)$, and in particular 
the same $H$-action, in either case.

 Then $(R^d\Gamma_{\fm_A}(A))^*$ obtains a right $H$-action.  Choosing a generator 
$\fe$ of degree $\bfl$ gives a map $\eta': H \to k$ given by $\fe^h = \fe \eta'(h)$ 
and the homological determinant is defined in \cite[Definition 3.3]{KKZ} to be $\eta = \eta' \circ S$.  In this paper, we use that the right $H$-action of  $(R^d\Gamma_{\fm_A}(A))^*$ induces a left $H$-action satisfying $\fe^h = S(h) \cdot \fe$ as in \eqref{E2.4.1}, since 
 $(R^d\Gamma_{\fm_A}(A))^*$ is a $H_{S^{-2}}$-equivariant bimodule.  Since $\eta \circ S^{2} = \eta$ by \eqref{E1.2.2}, the two definitions 
agree.  Moreover, it has already been commented in \cite[Remark 3.4]{KKZ} that the definition of homological determinant in \cite{KKZ} agrees in the case $H = kG$ is a group algebra with the original definition given in \cite{JoZ}.
\end{remark}

In general it is not easy to actually compute the homological determinant.  We mention
a few examples of connected graded AS Gorenstein algebras for which
the answer is known.

\begin{example}
\label{xxex3.8}
\begin{enumerate}
\item
If $A$ is the commutative polynomial ring $k[V]$ where $A_1=V$
is finite dimensional, then $\hdet \sigma=\det \sigma\mid_{V}$
for all $\sigma \in \Autz(A)$  \cite[p. 322]{JoZ}.
\item If $A$ is a graded down-up algebra (which is a special kind of AS regular algebra of global
dimension 3 generated by 2 degree 1 elements), then $\hdet \sigma=
(\det \sigma\mid_{A_1})^2$ by a result of Kirkman-Kuzmanovich
\cite[Theorem 1.5]{KK}.  
\item Let $A$ be the skew polynomial ring $k_{-1}[x,y]$ and let
$\sigma\in \Autz(A)$ map $x$ to $y$ and $y$ to $x$. Then
$\hdet \sigma=-\det \sigma\mid _{A_1}=1$.
\item If $A$ is $\mb{Z}^w$-graded and $\sigma$ is an automorphism which acts on each graded piece 
by a scalar, then we calculate $\hdet(\sigma)$ in Lemma~\ref{xxlem5.3} below.
\item Let $A$ be noetherian AS Gorenstein and
$\sigma \in \Autz(A)$. If $z$ is a normal nonzerodivisor such that
$\sigma(z)=\lambda z$ for some $\lambda\in k^\times$, then by
\cite[Proposition 2.4]{JiZ},
\begin{equation}
\label{E3.9.1}\tag{E3.9.1}
\hdet \sigma\mid_{A}=\lambda \hdet \sigma\mid_{A/(z)}.
\end{equation}
\end{enumerate}
\end{example}
\noindent On the other hand, it is unclear how to calculate $\hdet \sigma$ for an automorphism $\sigma$ 
of an arbitrary AS regular algebra.  

In the next result, we see that there is a useful restriction on the interaction between a Nakayama automorphism 
$\mu$ of a generalized AS Gorenstein algebra and a Hopf action on the algebra.

\begin{lemma}
\label{xxlem3.9}
Let $A$ be a generalized AS Gorenstein algebra which is a 
graded left $H$-module algebra for some Hopf algebra $H$.  Then $R^d\Gamma_{\fm_A}(A)^*\cong {^\mu A^1}(-{\bfl})$ is
naturally a graded $H_{S^{-2}}$-equivariant $A$-bimodule by Proposition~\ref{xxpro2.7}.  
 Assume further that
$R^d\Gamma_{\fm_A}(A)^*$ has an $H$-stable generator $\fe$, let  $\mu$ be the $\fe$-Nakayama automorphism, and let $\hdet$ be the $\fe$-homological determinant. 

Then the identity
\begin{equation}
\label{E3.10.1}\tag{E3.10.1} (\Xi^l_{\hdet} \circ S^{-2} )(h)(\mu(a))
=\mu(\Xi^r_{\hdet}(h)(a))
\end{equation}
holds for all $h\in H$ and all $a\in A$. As a consequence, if $H$ is
cocommutative or if $\hdet=\epsilon$, then
\begin{equation}
\label{E3.10.2}\tag{E3.10.2} S^{-2}(h)(\mu(a))=\mu(h(a))
\end{equation}
for all $a\in A$ and $h\in H$.
\end{lemma}

\begin{proof}
Write $\eta = \hdet: H\to k$ for convenience and recall that this
is an algebra map. Applying $h$ to $\fe\mu(a)=a\fe$ and using that
${^\mu A^1}(-{\bfl})$ is an $H_{S^{-2}}$-equivariant $A$-bimodule,
we have
$$\eta(h_1) \fe S^{-2}(h_2)(\mu(a))=
h_1(a)\eta(h_2)\fe=\fe \mu(h_1(a)\eta(h_2)).$$  This implies that
$$(\Xi^l_{\eta} \circ S^{-2})(h)(\mu(a)) =
\eta(h_1)S^{-2}(h_2)(\mu(a))= \mu(h_1(a)\eta(h_2))
=\mu(\Xi^r_{\eta}(h)(a)),$$
where we use $\eta \circ S^{-2} = \eta$ by \eqref{E1.2.2}.
This is \eqref{E3.10.1}.

If $\eta=\epsilon$ then $\Xi^l_{\eta}=\Xi^r_{\eta}=Id$, while if
$H$ is cocommutative, then $\Xi^l_{\eta}=\Xi^r_{\eta}$ and
$S^2 = 1$. In either case, \eqref{E3.10.2} is equivalent to
\eqref{E3.10.1}.  
\end{proof}

The following interesting result is an immediate corollary.

\begin{theorem}
\label{xxthm3.10}
Let $A$ be noetherian connected graded AS Gorenstein.
Then $\mu_A$ is in the center of $\Autz(A)$.
\end{theorem}

\begin{proof} This follows from applying Lemma~\ref{xxlem3.9} to
the action of $H = \Autz(A)$ on $A$, and noting that the group
algebra $H$ is cocommutative.
\end{proof}

For the rest of this section we consider Frobenius algebras.
Let $E$ be a finite dimensional ${\mathbb Z}^w$-graded algebra.
We say $E$ is a Frobenius algebra if there is a nondegenerate
associative bilinear form $\langle-, -\rangle: E\times E\to k$,
which is graded of degree ${-\bfl}\in {\mathbb Z}^w$.
This is equivalent to the existence of an isomorphism
$E^* \cong E[-\bfl]$ as graded left (or right) $E$-modules.
As a consequence, the injective dimension of $E$
is zero. The vector $\bfl\in {\mathbb Z}^w$
is called the AS index of $E$. There is a {\it classical Nakayama
automorphism} $\mu\in \Autw(E)$ such that $\langle a, b\rangle =
\langle \mu(b), a \rangle$ for all $a, b\in E$. Let
$\fe=\langle 1, -\rangle=\langle -, 1\rangle\in E^*$. Then $\fe$ is
a generator of $E^*$ such that $E^*\cong {^\mu E ^1}(-\bfl)$ as
graded $E$-bimodules. Note that $\mu$ and $\fe$ are dependent on the
choices of the bilinear form, so they are not necessarily unique.
See \cite{Mu} for further background.

It is easy to see that if $E$ is $\mb{N}$-graded with respect to the $\mid\mid \; \mid\mid$-grading, 
then $E$ is generalized AS Gorenstein, with $R^0 \Gamma_{\fm_E}(E)^* = E^* \cong {^\mu E ^1}(-\bfl)$.
Even if $E$ is not $\mb{N}$-graded, we still have that $E^* \cong {^\mu E ^1}(-\bfl)$; 
taking $G$ to be the subgroup of $\Autw(E)$ generated by $\mu$, then $E$ is naturally a
graded left $kG$-module algebra, and so $E^*$ is a graded $kG$-equivariant $E$-bimodule.
Then Definition~\ref{xxdef3.6} can be interpreted for $\fe \in E^*$ and we have the following.   
\begin{lemma}
\label{xxlem3.11}
Keep the notation above.   Then $\fe$ is $kG$-stable and
$\hdet \mu=1$, where $\hdet$ is the $\fe$-homological determinant.
\end{lemma}

\begin{proof}  For any $b\in E$,
$$\mu(\fe)(b)=\fe(\mu^{-1}(b))=\langle 1, \mu^{-1}(b)\rangle
=\langle b, 1\rangle=\fe(b)$$
which means that $\mu(\fe)=\fe$. Since $G=\langle \mu \rangle$,
$\fe$ is $G$-stable. The assertion follows by the definition of
$\hdet$ \eqref{E3.7.1}.
\end{proof}

\section{Proof of identity \eqref{HI1}}
\label{xxsec4}

The aim of this section is to prove homological  identity
\eqref{HI1} by computing the Nakayama automorphism of the smash
product of an AS Gorenstein algebra $A$ with a finite dimensional
Hopf algebra $H$ action.  This generalizes and partially recovers a
result of Le Meur \cite[Theorem 1]{LM}, who studies the differential
graded case, and a result of Liu-Wu-Zhu \cite[Theorem 2.12]{LiWZ},
where $A$ is assumed to be $N$-Koszul and $H$ is involutory (although
not necessarily finite-dimensional).  Other papers where similar
problems are considered include \cite{Fa, IR, WZhu}. Our approach differs
from the previous ones in several ways:  we do not assume finite global dimension, 
but rather the weaker Gorenstein condition, and
our methods emphasize the techniques of local cohomology.

\begin{theorem}
\label{xxthm4.1}
Suppose that $A$ is a generalized AS Gorenstein algebra of injective
dimension $d$ which is a graded left $H$-module algebra for a
finite-dimensional Hopf algebra $H$.   Let $B = A \# H$.  Suppose that there is an
$H$-stable generator $\fe\in R^d\Gamma_{\fm_A}(A)^*$.  Let $D_B$
and $D_A$ be the rigid dualizing complexes over $B$ and over $A$,
respectively.
\begin{enumerate}
\item
We have
$$D_B[-d] \cong
R^d\Gamma_{\fm_B}(B)^*  \cong H^* \# R^d\Gamma_{\fm_A}(A)^*
= {}^1H^\sigma \# D_A[-d],$$
as $B$-bimodules, where $\sigma =\mu_{H}^{-1} =S^2\circ \Xi^r_{\int^r}$.
\item
$B$ is generalized AS Gorenstein, with Nakayama automorphism
$$\mu_B=\mu_A\#(\Xi^l_{\hdet}\circ \mu_H),$$
where $\hdet$ is the $\fe$-homological determinant.
\item
If $A$ is connected graded AS regular and $H$ is semisimple, then
$B$ is skew CY.
\end{enumerate}
\end{theorem}

\begin{proof} 
(a) Recall that as mentioned in the proof of Lemma~\ref{xxlem3.5}, the results 
in \cite{VdB1} hold for not-necessarily connected but locally finite $\mb{N}$-graded algebras; in particular, 
they hold for the algebra $B$.  
Since $A$ is generalized AS Gorenstein, by \cite[Theorem 6.3]{VdB1} we have 
that $D_A[-d] = R^d\Gamma_{\fm_A}(A)^* \cong {}^{\mu} A^1(\bfl)$, as complexes concentrated 
in degree $0$, and $R^i\Gamma_{\fm_A}(A) = 0$ for $i \neq d$.

Now for any $i \geq 0$ we have
$$R^i\Gamma_{\fm_B}(B)^* =  R^i\Gamma_{\fm_A}(A \# H)^* =
(R^i\Gamma_{\fm_A}(A) \# H)^*$$
as $B$-bimodules, where we have used Lemma~\ref{xxlem3.2}, and the
fact that $A$ is an $H$-equivariant $A$-bimodule.    Thus $R^i\Gamma_{\fm_B}(B) = 0$ for $i \neq d$, 
and $R^d\Gamma_{\fm_B}(B)^* \cong (R^d\Gamma_{\fm_A}(A) \# H)^*$.  
Now we use Proposition~\ref{xxpro2.7} to identify
$$(R^d\Gamma_{\fm_A}(A) \# H)^*
\cong {}^1 H ^{\sigma} \# R^d\Gamma_{\fm_A}(A)^*
= {}^1 H ^{\sigma} \# D_A[-d]$$
as $B$-bimodules, where $\sigma = \mu_H^{-1} = S^2 \circ \Xi^r_{\int^r}$ by Lemma~\ref{xxlem2.6}.
Since $D_A[-d] = R^d\Gamma_{\fm_A}(A)^*$
is an $H_{S^{-2}}$-equivariant $A$-bimodule, the final term
is a well-defined $B$-bimodule as in Lemma~\ref{xxlem2.5}(b).

Note that since $H$ is finite-dimensional, $B$ is a finitely generated left and right $A$-module, so $B$ is noetherian also.  The algebra 
$B$ satisfies the $\chi$ condition and has finite cohomological dimension, since these properties also pass to a finite ring extension \cite[Theorem 8.3 and Corollary 8.4]{AZ}.  Thus 
the hypotheses of \cite[Theorem 6.3]{VdB1} also hold for $B$, so the rigid dualizing complex for $B$ exists and 
equals $R\Gamma_{\fm_B}(B)^*$.   Finally, this means that we have $D_B[-d]=R^d\Gamma_{\fm_B}(B)^*$.

(b) By Lemma~\ref{xxlem2.10}, taking $\fu$ to be a bimodule 
generator of $^1 H ^{\sigma}$, then $\fu \# \fe$ is a generator of
$^1 H ^{\sigma} \# {}^{\mu} A ^1(-\bfl)$, and we have
$$ {^1 H ^{\sigma}} \# {}^{\mu} A ^1(-\bfl) \cong {}^{\rho} B ^1$$
as $B$-bimodules, where $\rho = \mu_A\#(\Xi^l_{\hdet}\circ
\sigma^{-1})$. This implies that $D_B \cong {^\rho B^1}[d](-\bfl)$, and thus $\mu_B = \rho$ has 
the claimed formula, and 
we have verified Definition \ref{xxdef3.3}(c) for $B$.  We already checked in part (a) that $B$ is noetherian, satisfies $\chi$, 
and has finite cohomological dimension.
By definition a dualizing complex has finite injective dimension in the derived category of left or right modules 
\cite[Definition 6.1]{VdB1}.  Since the dualizing complex for $B$ is isomorphic to (a shift of) $B$ as a right or left $B$-module, 
it follows that $B$ has finite injective dimension on both sides.  Thus all parts of the definition of generalized AS Gorenstein hold 
for $B$.

(c) If $A$ is AS regular and $H$ is semisimple, then $A$ is homologically smooth by Lemma~\ref{xxlem1.2} and $H$ is 
homologically smooth since it is semisimple.   By \cite[Proposition 2.11]{LiWZ}, $B=A\# H$ is
homologically smooth.  The rest of the definition of skew CY for $B$ follows from part (b) and Lemma~\ref{xxlem3.5}.  
\end{proof}

Theorem \ref{xxthm0.2} (namely, \eqref{HI1}) follows immediately from the previous result.
As mentioned earlier, one of the motivations behind our study of the result above was to 
better understand examples such as  Example \ref{xxex1.7}, where $A$ is only skew-CY, but some skew group algebra $A \rtimes G$ 
becomes CY.  The following corollary, which gives a special case of  Corollary~\ref{xxcor0.6}, explains this phenomenon.  
\begin{corollary}
\label{xxex4.2}
\label{xxcor4.2}
Let $A$ be noetherian connected $\mb{N}$-graded AS regular algebra with Nakayama automorphism $\mu$.
Suppose that $\mu$ has finite order.   If $\sigma$ is a graded algebra
automorphism of $A$ such that $\sigma^n=\mu$ for some $n$, and $\hdet \sigma=1$, then $A \# kG$ is CY, 
where $G = \langle \sigma \rangle$ is the subgroup of $\Autz(A)$ generated by $\sigma$.
\end{corollary}
\begin{proof}
Note that $\sigma$ has finite order. Let $H = kG$, and let $B = A \# H$, which is 
the same as the skew group algebra  $A\rtimes G$.  The algebra $B$ is skew CY by Theorem~\ref{xxthm4.1}(c). 
By Theorem \ref{xxthm4.1}(b), $\mu_B=
\mu_A\# (\Xi^l_{\hdet} \circ \mu_H)$. Since $\hdet \sigma=1$, the homological determinant 
$\hdet: kG \to k$ is trivial, and as a consequence, $\Xi^l_{\hdet}= Id_H$.
Since $H$ is semisimple, $\mu_H=Id_H$ \cite[1.7(a)]{Lo}.  Thus $\mu_B=\mu_A\# Id=
\sigma^{n}\# Id$, which is inner since it is conjugation in $B$ by 
$1\# \sigma^n$.  So $B$ is CY.
\end{proof}

One interesting open question is whether Theorem \ref{xxthm4.1}(c) holds in
a more general setting.

\begin{question}
\label{xxque4.3}
Let $H$ be a Hopf algebra and let $A$ be a left $H$-module algebra,
neither of which is necessarily graded. If $A$ and $H$ are skew
CY, then is $A \# H$ skew CY?  If so, what is the Nakayama automorphism
$\mu_{A \# H}$, in terms of $\mu_A$ and $\mu_H$?
\end{question}
\noindent

\section{Proof of identity \eqref{HI2}}
\label{xxsec5}
The goal of this section is to prove homological identity
\eqref{HI2}, and we consider a slightly more general setting.
Let $A$ be a ${\mathbb Z}^{w}$-graded generalized AS Gorenstein algebra.
Let $G$ be a subgroup of $\Autw(A)$, so that $A$ is a left $kG$-module
algebra where $\sigma(a)$ has its usual meaning for $\sigma \in G
\subseteq \Autw(A)$. In this context, a $kG$-eqivariant $A$-bimodule is
an $A$-bimodule $M$ with left $G$-action, denoted by $\alpha_\sigma:
M \to M$ for each $\sigma\in G$, such that
\begin{equation}
\label{E5.0.1}\tag{E5.0.1} \alpha_{\sigma}(a m b)=\sigma(a)
\alpha_\sigma(m)\sigma(b)
\end{equation}
for all $a,b\in A$, all $m\in M$ and all $\sigma\in G$.
For a fixed $\sigma$, any morphism $\alpha$ of $A$-bimodules satisfying
\eqref{E5.0.1} is also called a $\sigma$-linear $A$-bimodule morphism.

We review the definition of  graded twists of ${\mathbb Z}^w$-graded
algebras and  ${\mathbb Z}^w$-graded modules \cite{Zh}. For simplicity,
we only consider the graded twists by automorphisms of the algebra. Let
$\sigma:= \{\sigma_1,\cdots, \sigma_w\}\subset \Autw(A)$ be a sequence of
commuting ${\mathbb Z}^w$-graded automorphisms of $A$. Recall that $|m|$
denotes the ${\mathbb Z}^w$-degree of a homogeneous element $m$ in a
${\mathbb Z}^w$-graded module $M$.  Let $v$ be an integral vector
$(v_1,\cdots,v_w)$. Write $\sigma^v=\sigma_1^{v_1}\cdots \sigma_w^{v_w}$.
Define the twisting system associated to $\sigma$ to be the set
$$\tilde{\sigma}= \{\sigma^v \mid v\in {\mathbb Z}^w\}.$$
A (left) graded twist of $A$
associated to $\tilde{\sigma}$ is a new graded algebra, denoted by
$A^{\tilde{\sigma}}$, such that $A^{\tilde{\sigma}}=A$ as
a ${\mathbb Z}^w$-graded vector space, and where the new multiplication
$\cc$ of $A^{\tilde{\sigma}}$ is given by
\begin{equation}
\label{E5.0.2}\tag{E5.0.2} a \cc b =\sigma^{|b|}(a) b
\end{equation}
for all homogeneous elements $a,b \in A$.  We note that
the paper \cite{Zh} works primarily with right graded twists, but 
left graded twists are more convenient in our setting.
Given a left graded $A$-module $N$, a left graded twist of $N$ is
defined by the same formula \eqref{E5.0.2} for all homogeneous
$a \in A$, $b\in N$ and denoted by $N^{\tilde{\sigma}}$.  Then 
$N^{\tilde{\sigma}}$ is naturally a left $A^{\tilde{\sigma}}$-module, and the functor $N \mapsto
N^{\tilde{\sigma}}$ gives an equivalence of graded module categories
$A$-$\GrMod \simeq  A^{\tilde{\sigma}}$-$\GrMod$ \cite[Theorem 3.1]{Zh}.

Next, we define the left twist of a graded $kG$-equivariant $A$-bimodule $M$.
Continue to write $\sigma=\{\sigma_1,\cdots, \sigma_w\}$, and assume
now that each $\sigma_i$ is in the center of $G$.  (Since the $\sigma_i$
are assumed to pairwise commute, this additional assumption can
be effected if necessary by replacing $G$ with the subgroup generated by
the $\sigma_i$.) If $v$ is an integral vector $(v_1,\cdots,v_w)$,
write $\alpha_{\sigma}^v$ for $\prod_{s=1}^w \alpha_{\sigma_s}^{v_s}=
\alpha_{\sigma^v}$. The left graded twist of $M$ associated to
$\sigma$, denoted by $M^{\tilde{\sigma}}$, is defined as follows: as a
${\mathbb Z}^w$-graded $k$-space, $M^{\tilde{\sigma}}=M$, and the left and
right $A^{\tilde{\sigma}}$-multiplication is defined by
\begin{equation}
\label{E5.0.3}\tag{E5.0.3}
a\cc m \cc b=\sigma^{|m|+|b|}(a) \alpha_{\sigma}^{|b|}(m) b
\end{equation}
for all homogeneous elements $a,b \in A=A^{\tilde{\sigma}}$ and $m\in M$.
It is routine to check that $M^{\tilde{\sigma}}$ is a left $kG$-equivariant
$\mb{Z}^w$-graded $A^{\tilde{\sigma}}$-bimodule, where $g \in G$ acts
by the same map $\alpha_g: M \to M$ of the underlying $k$-space.  It is
also easy to check that the bimodule twist is functorial, in the sense
that if $\phi: M \to N$ is a graded $kG$-equivariant $A$-bimodule map, then the same underlying
set map gives a $kG$-equivariant  $A^{\tilde{\sigma}}$-bimodule map $\phi:
M^{\tilde{\sigma}} \to N^{\tilde{\sigma}}$.

\begin{lemma}
\label{xxlem5.1}
Assume that $A$ is finitely graded with respect to the $\mid\mid \;\;\mid\mid$-grading.
Let $M$ be a left $kG$-equivariant ${\mathbb Z}^w$-graded $A$-bimodule.
For any $d \geq 0$, $R^d\Gamma_{\fm_{A^{\tilde{\sigma}}}}
(M^{\tilde{\sigma}}) \cong R^d \Gamma_{\fm_A}(M)^{\tilde{\sigma}}$ as
graded left $kG$-equivariant $A^{\tilde{\sigma}}$-bimodules.
\end{lemma}

\begin{proof}
The case $d = 0$ is easy.   Namely, $\Gamma_{\fm_{A^{\tilde{\sigma}}}}
(M^{\tilde{\sigma}}) = \Gamma_{\fm_A}(M)^{\tilde{\sigma}}$ since both can be  
identified with the same subset of $M^{\tilde{\sigma}}$.

Now if $R = A^e \rtimes kG$ is the Kaygun algebra of
Remark~\ref{xxrem2.3}, we may find an injective resolution
$M\to I^\bullet$  in the category of $\mb{Z}^w$-graded left $R$-modules.  
As we have noted in the proof of Lemma~\ref{xxlem3.2}, this is also a graded injective left
$A$-module resolution.  Now since the bimodule twist is functorial,
we can twist the entire complex as in \eqref{E5.0.3} to get an
exact complex $M^{\tilde{\sigma}} \to (I^\bullet)^{\tilde{\sigma}}$
of $kG$-equivariant graded $A^{\tilde{\sigma}}$-bimodules, where the
maps are the same underlying vector space maps.  Since the bimodule
twist \eqref{E5.0.3} restricts to the usual left twist as left
$A$-modules, $(I^\bullet)^{\tilde{\sigma}}$ is a a graded injective resolution
of $M^{\tilde{\sigma}}$ as left $A^{\tilde{\sigma}}$-modules by
\cite[Theorem 3.1]{Zh}, so we can use it to calculate
$R^d\Gamma_{\fm_{A^{\tilde{\sigma}}}}(M^{\tilde{\sigma}})$.  Since
$\Gamma_{\fm_{A^{\tilde{\sigma}}}}((I^i)^{\tilde{\sigma}})
=\Gamma_{\fm_A}(I^i)^{\tilde{\sigma}}$ for each $i$, the result follows.
\end{proof}

\begin{lemma}
\label{xxlem5.2}
Let $M$ be a $kG$-equivariant graded $A$-bimodule.  Then there is an
isomorphism $(M^{\tilde{\sigma}})^* \cong  (M^*)^{\tilde{\sigma}}$
of left $kG$-equivariant graded $A^{\tilde{\sigma}}$-bimodules.
\end{lemma}
\begin{proof}
Note that $M^*$ is naturally a $kG$-equivariant graded $A$-bimodule by
Proposition~\ref{xxpro2.7} (since $S^2 = 1$), so $(M^*)^{\tilde{\sigma}}$ is 
a $kG$-equivariant graded $A^{\tilde{\sigma}}$-bimodule. Similarly,
$(M^{\tilde{\sigma}})^*$ is a $kG$-equivariant graded
$A^{\tilde{\sigma}}$-bimodule.

By definition, as graded $k$-vector spaces,
$$(M^{\tilde{\sigma}})^*=M^*,$$
and we identify these below. We calculate the left and right
$A^{\tilde{\sigma}}$-action on $(M^{\tilde{\sigma}})^*$. Let $x\in (M^{\tilde{\sigma}})^*$,
$m \in M$ and $a\in A^{\tilde{\sigma}}$. Let $\circ$ denote the left and
right $A^{\tilde{\sigma}}$-actions on $(M^{\tilde{\sigma}})^*$, and let
$\langle \quad, \quad \rangle$ be the canonical bilinear form
$M^* \times M\to k$. Then
$$
\begin{aligned}
\langle x \circ a, m\rangle&=\langle x, a \cc m\rangle
= \langle x, \sigma^{|m|}(a)m\rangle
=\langle x\sigma^{|m|}(a), m\rangle\\
&=\langle x\sigma^{-|x|-|a|}(a), m\rangle,
\qquad {\text{since both sides are zero if $|m|+|x|+|a|\neq 0$.}}
\end{aligned}
$$
Hence $x \circ a=x\sigma^{-|x|-|a|}(a)$.
For the left action, we have
$$\begin{aligned}
\langle a\circ x, m\rangle&=\langle x, m\cc a \rangle
= \langle x, \alpha_{\sigma}^{|a|}(m) a\rangle
=\langle x, \alpha_{\sigma}^{|a|}(m \sigma^{-|a|}(a))\rangle\\
&=\langle (\alpha^*_{\sigma})^{|a|}(x),m \sigma^{-|a|}(a)\rangle
=\langle \sigma^{-|a|}(a) (\alpha^*_{\sigma})^{|a|}(x),m \rangle.
\end{aligned}
$$
The natural $kG$-action on $M^*$, denoted by $\alpha_{\gamma}\mid_{M^*}$
for all $\gamma\in G$, satisfies
$$\alpha_{\gamma}\mid_{M^*}=(\alpha_{S(\gamma)})^*
=[(\alpha_{\gamma})^*]^{-1}$$ since $S(\gamma)=\gamma^{-1}$.
Then
$$\langle \sigma^{-|a|}(a) (\alpha^*_{\sigma})^{|a|}(x),m \rangle=
\langle \sigma^{-|a|}(a) (\alpha_{\sigma}\mid_{M^*})^{-|a|}(x),m
\rangle$$ and consequently,
$$a\circ x=\sigma^{-|a|}(a) (\alpha_{\sigma}\mid_{M^*})^{-|a|}(x).$$

Combining the two calculations above, it follows that the
$A^{\tilde{\sigma}}$-bimodule $(M^{\tilde{\sigma}})^*$ has left
and right $A^{\tilde{\sigma}}$ actions satisfying the rule
\begin{equation}
\label{E5.2.1}\tag{E5.2.1} a\circ n \circ  b=
\sigma^{-|a|}(a)\alpha_{\sigma}^{-|a|}(n) \sigma^{-|a|-|n|-|b|}(b)
\end{equation}
for all homogeneous elements $a,b \in A=A^{\tilde{\sigma}}$ and
homogeneous element $n\in N = M^*$. It is now straightforward to
check that the function $\phi: n \to \alpha_{\sigma}^{-|n|}(n)$
satisfies $\phi( a \cc n \cc b) = a \circ \phi(n) \circ b$ for all
homogeneous $n \in N$, $a, b \in A$.  Thus $\phi$ gives an isomorphism
from $(M^*)^{\tilde{\sigma}}$ to $(M^{\tilde{\sigma}})^*$
(identifying the underlying $k$-space of each with $M^*$).  It is
also clear that $\phi$ respects the left $kG$-action, since the $\sigma_i$ are in the center of $G$.  Thus 
$\phi$ is the required isomorphism of $kG$-equivariant $A^{\tilde{\sigma}}$-bimodules.
\end{proof}

For any $\delta=(\delta_1,\cdots,\delta_w)\in (k^\times)^{w}$ and
any $v=(v_1,\cdots,v_w)\in {\mathbb Z}^w$, write $\delta^v$ for
$\prod_{s=1}^w \delta_s^{v_s}$. Given a $\delta\in (k^\times)^{w}$,
a graded algebra automorphism $\xi_{\delta}$ of $A$ is defined by
$$\xi_{\delta}(a)=\delta^{|a|} a$$
for all homogeneous elements $a\in A$. Note that $\xi_{\delta}$
is in the center of $\Autw(A)$. If $\delta=(c,c,\cdots,c)$, then
$\xi_{\delta}=\xi_c$ as defined in the introduction.

\begin{lemma}
\label{xxlem5.3}
Let $A$ be $\mb{Z}^w$-graded generalized AS Gorenstein and suppose
that $G$ is some subgroup of $\Autw(A)$. Retain the notation as above.
\begin{enumerate}
\item
Suppose $G$ consists of automorphisms of the form $\xi_\delta$.
Then every homogeneous generator $\fe$ of $R^d\Gamma_{\fm_A}(A)^*$
is $kG$-stable and the $\fe$-homological determinant satisfies
$\hdet \xi_{\delta}=\delta^{\bfl}$.
\item
If $R^d\Gamma_{\fm_A}(A)^*$ has a $kG$-stable generator $\fe$,
then every element of $G$ commutes with the $\fe$-Nakayama
automorphism $\mu$ (in $\Autw(A)$).
\end{enumerate}
\end{lemma}

\begin{proof}
(a) Every $\mb{Z}^w$-graded $A$-bimodule $M$ has a canonical
$kG$-equivariant bimodule structure where $\xi_{\delta}(x) =
\delta^{|x|} x$ for all homogeneous $x$.   If we take any graded 
injective resolution of $I^{\bullet}$ of $A$ as left $A^e$-modules, then 
making each $I^i$ into a $kG$-equivariant bimodule in the canonical way, 
it is easy to see that the morphisms in the complex automatically respect 
the $kG$ action.  Then $I^{\bullet}$ is already a graded $R$-module 
resolution of $A$, where $R = A^e \rtimes kG$ is the Kaygun algebra, and so $I^{\bullet}$ is also an injective 
resolution in $A-\GrMod$.  As noted in the proof of Lemma~\ref{xxlem3.2}, we can 
use $I^{\bullet}$ to calculate the $R$-module structure on $M:=R^d\Gamma_{\fm_A}(A)$.  
Thus we see that the induced $kG$-action on $M$ is also the canonical one.  

For an element $g \in G$, the left $g$-action on
$R^d\Gamma_{\fm_A}(A)^*$ is induced by the right $g$ action on
$R^d\Gamma_{\fm_A}(A)$, which is the same as the left $S(g)
= g^{-1}$-action.  Since also elements of $\Hom_k(M_v, k)$ have
degree $-v$, we see that the left $kG$-action on
$R^d\Gamma_{\fm_A}(A)^*$ is also the canonical one.
Now let $\fe$ be any homogeneous generator of $(R^d\Gamma_{\fm}(A))^*$.
Then the degree of $\fe$ is $\bfl$ and the $kG$ action on $\fe$ is
the canonical one given by
$$\xi_{\delta} \cdot (\fe) = \delta^{\bfl} \fe$$
for each $\xi_{\delta} \in G$.
Thus $\fe$ is $G$-stable and $\hdet \xi_{\delta}=\delta^{\bfl}$.

(b) This is a special case of \eqref{E3.10.2}.
\end{proof}

We are now ready to prove our main result determining the Nakayama
automorphism of a graded twist.

\begin{theorem}
\label{xxthm5.4}
Let $A$ be $\mb{Z}^w$-graded generalized AS Gorenstein of AS index $\bfl$.
Let $G \subset \Autw(A)$, and assume that $\fe$ is a $kG$-stable generator 
of $R^d\Gamma_{\fm_A}(A)^*$.   Let $\hdet_A$ be the $\fe$-homological determinant, and let $\mu_A$
be the $\fe$-Nakayama automorphism.   Suppose that $\sigma = (\sigma_1, \dots, \sigma_w)$
is a collection of graded automorphisms in the center of $G$, and let $\hdet(\sigma)$ denote the
vector $(\hdet(\sigma_1),\cdots, \hdet(\sigma_w))\in (k^\times)^w$.   For convenience of notation write $\fm = \fm_A$ and
$\fm^{\tilde{\sigma}} = \fm_{A^{\tilde{\sigma}}}$. 
\begin{enumerate}
\item  
The algebra $A^{\tilde{\sigma}}$ is generalized AS Gorenstein.
Under the isomorphism 
\[
\psi: (R^d\Gamma_{\fm^{\tilde{\sigma}}}(A^{\tilde{\sigma}}))^* \cong (R^d\Gamma_{\fm}(A)^{\tilde{\sigma}})^* \cong  ((R^d\Gamma_{\fm}(A))^*)^{\tilde{\sigma}}
\]
coming from Lemmas~\ref{xxlem5.1} and \ref{xxlem5.2}, $\fe' = \psi^{-1}(\fe)$ is a $kG$-stable generator of 
$(R^d\Gamma_{\fm^{\tilde{\sigma}}}(A^{\tilde{\sigma}}))^*$.  If $\mu_{A^{\tilde{\sigma}}}$ 
is the $\fe'$-Nakayama automorphism, then 
$$ \mu_{A^{\tilde{\sigma}}}=\mu_A\circ \sigma^{\bfl}\circ
\xi_{\hdet(\sigma)}^{-1}.
$$  
\item
Let $\hdet_{A^{\tilde{\sigma}}}$
be the $\fe'$-homological determinant.  Then
$\hdet_A(\tau) = \hdet_{A^{\tilde{\sigma}}}(\tau)$ for all $\tau \in G$.
\item
If $\mu_A\in G$, then $\hdet_{A^{\tilde{\sigma}}} \mu_{A^{\tilde{\sigma}}}
=\hdet_A \mu_A$.
\end{enumerate}
\end{theorem}

\begin{proof} (a)  Since the isomorphism 
$(R^d\Gamma_{\fm^{\tilde{\sigma}}}(A^{\tilde{\sigma}})) \cong (R^d\Gamma_{\fm}(A)^{\tilde{\sigma}})$
in Lemma~\ref{xxlem5.1} is an isomorphism of $kG$-equivariant bimodules, taking duals we also get an isomorphism 
respecting the $kG$-action.  The isomorphism of Lemma~\ref{xxlem5.2} also respects the $kG$-action, so both 
isomorphisms making up $\psi$ are isomorphisms of graded $kG$-equivariant bimodules.  Thus $e'$ must be a $kG$-stable generator
of $(R^d\Gamma_{\fm^{\tilde{\sigma}}}(A^{\tilde{\sigma}}))^*$, of degree $\bfl$.

By the generalized AS Gorenstein property for $A$, we have  $R^d\Gamma_{\fm}(A)^* \cong {^{\mu_{A}} A^{1}}(-\bfl)$, 
and so ${^{\mu_{A}} A^{1}}(-\bfl)$ has a $kG$-stable generator corresponding to $\fe \in R^d\Gamma_{\fm}(A)^*$, which we also call $\fe$.  Combining this isomorphism with $\psi$ we have 
$$
(R^d\Gamma_{\fm^{\tilde{\sigma}}}(A^{\tilde{\sigma}}))^*
\cong ({^{\mu_{A}} A^{1}}(-\bfl))^{\tilde{\sigma}},
$$
where $\fe'$ corresponds to the $kG$-stable generator $\fe$ of $(^{\mu_{A}} A^{1}(-\bfl))^{\tilde{\sigma}}$.

We calculate the left and right action on
$(^{\mu_{A}} A^{1}(-\bfl))^{\tilde{\sigma}}$. For every
$a\in A^{\tilde{\sigma}}$, we have
$$a\cc \fe =\sigma^{\bfl}(a) \fe   = \fe \mu_A(\sigma^{ \bfl}(a))$$
and
$$\fe \cc a= \alpha_{\sigma}^{|a|}(\fe) a =\alpha_{\sigma^{|a|}}(\fe) a
= \fe \hdet({\sigma}^{|a|}) a$$
since by definition the action of $\alpha_{\sigma}$ on $\fe$ is given by the homological determinant.
Hence $(^{\mu_{A}} A^{1}(-\bfl))^{\tilde{\sigma}}$
is isomorphic to ${^\phi (A^{\tilde{\sigma}})^1}$ where
$\phi(\hdet({\sigma}^{|a|}) a) = \mu_A(\sigma^{ \bfl}(a)).$
Thus
$$\phi= \mu_A \circ
\sigma^{\bfl}\circ \xi_{\hdet(\sigma)}^{-1}.$$
The formula for $\mu_{A^{\tilde{\sigma}}}$ follows.  In particular, part (c) of the definition of generalized AS Gorenstein 
holds for $A^{\tilde{\sigma}}$, and the other parts of the definition are easy to check using the properties of 
graded twists \cite{Zh}.  

(b)  
This follows immediately from the fact that $\fe'$ and $\fe$
correspond under an isomorphism of $kG$-equivariant bimodules.

(c) Since $\xi_{\hdet(\sigma)}$ is in the center of $\Autw(A)$ and $\fe$
is always $\xi_{\hdet(\sigma)}$-stable, we can
replace $G$ with the group generated by $\xi_{\hdet(\sigma)}$ and $G$ without changing the hypotheses. Then
by parts (a,b) and the fact that 
\[
\hdet(\sigma^{\bfl} \circ \xi_{\hdet(\sigma)}^{-1})
= (\hdet \textstyle{\prod}_s \sigma_s^{\bfl_s})(\textstyle{\prod}_s \hdet(\sigma_s)^{\bfl_s})^{-1}
=1,
\]
the assertion follows.
\end{proof}

The proof of homological identity \eqref{HI2} is now an easy special case of the preceding theorem.
\begin{proof}[Proof of Theorem \ref{xxthm0.3}]
Let $G$ be the subgroup of $\Aut(A)$ generated by $\sigma$.
Then $G$ is abelian, so in particular $\sigma$ is in the center of $G$.
Since $A$ is connected graded, any generator $\fe$ of
$R^d \Gamma_{\fm_A}(A)^*$ is $G$-stable.  The assertion follows from
Theorem \ref{xxthm5.4}(a).
\end{proof}

We close this section with a simple example.
\begin{example}
\label{xxex5.5}
Let $A:=k_{p_{ij}} [x_1,\cdots,x_w]$ be the skew polynomial ring
generated by $x_1,\cdots,x_n$ subject to the relations
$$x_jx_i=p_{ij}x_ix_j$$
for all $i<j$, where $\{p_{ij}\}_{1\leq i<j\leq w}$ is a set of
nonzero scalars. Then one checks immediately that $A$ is isomorphic to 
the ${\mathbb Z}^w$-graded twist of the
commutative polynomial ring $B = k[x_1, \dots, x_w]$ by  $\sigma=
(\sigma_1,\cdots,\sigma_w)$ where $\sigma_i$ is
defined by 
\[
\sigma_i(x_s)=\begin{cases}p_{is}  x_s & i<s\\
x_s & i\geq s\end{cases}, \ \ \  \text{for all}\ s. 
\]

It is easy to see that $\bfl = (1, 1, \dots, 1)$, and of course $\mu_B = 1$.  
We calculate that $\sigma^{\bfl}(x_s) = \prod_i \sigma_i(x_s) = \prod_{a < s} p_{as} x_s$ 
and using Lemma~\ref{xxlem5.3}(a), we get that 
$\xi_{\hdet(\sigma)}(x_s) = \hdet(\sigma_s) x_s = \prod_{b > s}p_{sb} x_s$.  Combining these 
calculations, by Theorem \ref{xxthm5.4}(a) we have 
$$\mu_A(x_s)=(\prod_{a<s} p_{as}\prod_{b>s}p_{sb}^{-1})x_s$$
for all $s=1,2,\cdots,w$.

If $w=2$, $k_{p_{12}}[x_1,x_2]$ is a ${\mathbb Z}$-graded twist of
$k[x_1,x_2]$. For $w\geq 3$, $k_{p_{ij}} [x_1,\cdots,x_w]$ is not
in general a ${\mathbb Z}$-graded twist of any commutative ring. In
fact,  it is easy to check that $k_{p_{ij}} [x_1,\cdots,x_w]$ is a
${\mathbb Z}$-graded twist of $B$ if and only if there is a set of 
nonzero scalars $\{p_1,\cdots,p_w\}$ such that $p_{ij}=p_i p_j^{-1}$ for all $i, j$.  In particular, 
this example demonstrates why it is useful to consider the generality of $\mb{Z}^w$-graded twists 
as we have done in this section.
\end{example}

\section{Proof of Identity \eqref{HI3}}
\label{xxsec6}

The goal of this section is to prove homological identity
\eqref{HI3}, which we  recall states that the homological determinant of the Nakayama automorphism 
is $1$.  We only consider the connected graded case for simplicity, and we prove it 
only for Koszul AS regular algebras in this paper.  We do show, however, how our formula 
for the Nakayama automorphism of a graded twist from the previous section allows one to reduce 
to the case of an algebra with simple Nakayama automorphism of the form $\xi_c$.  This same method may be useful 
to prove the result in general.

We start with a general lemma about tensor products of Gorenstein algebras.
Let $\bfl(A)$ denote the AS index of an AS Gorenstein algebra $A$.
\begin{lemma}
\label{xxlem6.1} Let $A$ and $B$ be noetherian connected graded
AS Gorenstein algebras.  Suppose that $A\otimes B$ is noetherian.
\begin{enumerate}
\item
The algebra $A \otimes B$ is AS Gorenstein and we have $\mu_{A\otimes B}= \mu_A\otimes \mu_B$ and
$\bfl(A\otimes B)=\bfl(A)+\bfl(B)$ as $\mb{N}$-graded algebras.  The algebra $A \otimes B$ is also $\mb{Z}^2$-graded, 
where the $(i,j)$ graded piece is $A_i \otimes B_j$, and with this grading, $\bfl(A \otimes B) = (\bfl(A), \bfl(B))$. 
\item
Let $\sigma\in \Autz(A)$, $\tau\in \Autz(B)$. Then
$\hdet_{A\otimes B}(\sigma\otimes \tau)=(\hdet_A \sigma)( \hdet_B \tau)$.
\end{enumerate}
\end{lemma}

\begin{proof} (a)  Let $d_1=\injdim A$
and $d_2=\injdim B$.   By \cite[Theorem 7.1]{VdB1},
\begin{equation}
\label{E6.1.1}\tag{E6.1.1}
R^{d_1+d_2}\Gamma_{\fm_{A\otimes B}}(A\otimes B)
\cong R^{d_1}\Gamma_{\fm_A}(A)\otimes R^{d_2}\Gamma_{\fm_B}(B), 
\end{equation}
and this is a rigid dualizing complex for $A \otimes B$.  
Since $A$ and $B$ are AS Gorenstein, we have $R^{d_1}\Gamma_{\fm_A}(A)^* \cong {}^{\mu_A} A ^1(-\bfl(A))$ 
and $R^{d_2}\Gamma_{\fm_B}(B)^* \cong {}^{\mu_B} B ^1(-\bfl(B))$, 
and thus $R^{d_1+d_2}\Gamma_{\fm_{A\otimes B}}(A\otimes B)^* \cong {}^{\mu_A \otimes \mu_B} (A \otimes B)^1(-\bfl(A) -\bfl(B))$.     Now we note that $A \otimes B$ satisfies  Definition~\ref{xxdef3.3}, since 
the $\chi$ condition for $A \otimes B$ is part of the existence of the dualizing complex (\cite[Theorem 6.3]{VdB1}), and $A \otimes B$ 
must have finite injective dimension since a dualizing complex always does by definition.   Thus $A \otimes B$ is 
generalized AS Gorenstein and so is also AS Gorenstein in the usual sense by Remark~\ref{xxrem3.5}.  We have already 
calculated its Nakayama automorphism and AS index above.  The proof of the multigraded result is the same, since \eqref{E6.1.1} 
also holds as $\mb{Z}^2$-graded modules.  

(b) 
Let $G_1=\langle \sigma\rangle$ and $G_2=\langle \tau\rangle$.
Then $G = G_1\times G_2$ is naturally a subgroup of $\Autz(A\otimes B)$.   Recall that the 
$kG$ structure on $R^{d_1+d_2}\Gamma_{\fm_{A\otimes B}}(A\otimes B)$ can 
be calculated using an $(A \otimes B) \# kG$-injective resolution of $A \otimes B$, as in Remark~\ref{xxrem3.7}.  
If $I^{\bullet}$ is an $A \# kG_1$ graded injective resolution of $A$ and $J^{\bullet}$ is an 
$B \# kG_2$ graded injective resolution of $B$, then the tensor product complex $I^{\bullet} \otimes_k J^{\bullet}$ 
is a resolution of $A \otimes B$ in the category of graded $(A \otimes B) \# k(G_1 \times G_2)$-modules.
By \cite[Lemma 4.5]{VdB1}, we have $R\Gamma_{\fm_{A \otimes B}} = R\Gamma_{\fm_A} \circ R\Gamma_{\fm_B}$.
Also, since we assume $A$ and $B$ are noetherian, direct sums of injective modules over these rings are injective.  
Then each term of $I^{\bullet} \otimes_k J^{\bullet}$ is an injective $B$-module.  Applying $\Gamma_{\fm_B}$ the resulting complex $I^{\bullet} \otimes \Gamma_{\fm_B}(J^{\bullet})$ consists of injective $A$-modules and so applying $\Gamma_{\fm_A}$ 
we get  that $\Gamma_{\fm_A}(I^{\bullet}) \otimes \Gamma_{\fm_B}(J^{\bullet})$ is equal to $R\Gamma_{\fm_{A \otimes B}}(A \otimes B)$.
Now by construction, the action of $G_1 \times G_2$ on the complex 
\[
\Gamma_{\fm_A}(I^{\bullet}) \otimes \Gamma_{\fm_B}(J^{\bullet}) \cong {}^{\mu_A} A ^1(-\bfl(A))[-d_1] \otimes {}^{\mu_B} B ^1(-\bfl(B))[-d_2]
\]
simply comes from the obvious action induced by $G_1$ acting on the first tensor component and $G_2$ acting on the second tensor component.  Taking $k$-linear 
duals, the result follows from the definition of $\hdet$.
\end{proof}

We now show how to pass from an arbitrary AS Gorenstein algebra to a closely related one with simpler Nakayama automorphism.
\begin{lemma}
\label{xxlem6.2}
Let $A$ be a connected graded AS Gorenstein algebra.  Possibly after replacing the base field $k$ with a finite extension, there is a 
connected graded AS Gorenstein algebra $B$ such that $\hdet \mu_A = \hdet \mu_B$, and where $\mu_B = \xi_c$ for some $c \in k$.   Moreover, $B$ is a multigraded twist of a commutative polynomial extension of $A$.
\end{lemma}
\begin{proof} 
Consider first a noetherian connected ${\mathbb N}^2$-graded 
generalized AS Gorenstein algebra $C$ with AS index ${\bfl}=(l,1)\neq (-1,1)$, and 
let $\mu_C$ be the Nakayama automorphism of $C$.  
Define 
$\sigma=\{Id, \mu_C^{-1}\}$. By Theorem \ref{xxthm5.4}(a), 
$\mu_{C^{\tilde{\sigma}}}=\xi^{-1}_{\hdet \sigma}
=\xi_{1,c}$ where $c=\hdet \mu_C$.  
Since $l \neq -1$, write $c=d^{l+1}$
for some $d\in k$, which we can do after replacing $k$ by a finite extension field if necessary.  (It is not hard to see that
all properties of the algebras are preserved by a finite extension of the base field.) 
Since $\hdet$ is independent of the grading, we now view 
$D = C^{\tilde{\sigma}}$ as a connected ${\mathbb N}$-graded algebra.  Let $\tau=\xi_{1,d^{-1}}$.  Then by Theorem \ref{xxthm5.4}(a), $\mu_{D^{\tilde{\tau}}}=\xi^{-1}_{\hdet \tau}=\xi_{c'}$ for 
some $c'\in k^\times$.   Finally, by Theorem~\ref{xxthm5.4}(c), $\hdet \mu_{D^{\tilde{\tau}}} = \hdet \mu_D = \hdet \mu_C$.

Now suppose that $A$ is a noetherian connected graded AS Gorenstein algebra 
with AS index $\bfl \neq -1$.  Let $C=A\otimes k[x]$, which is a noetherian
${\mathbb Z}^2$-graded generalized AS Gorenstein algebra with AS index
$(\bfl,1)\neq (-1,1)$, by Lemma~\ref{xxlem6.1}(a).  By the previous paragraph, there is an algebra $B$ 
(which is a sequence of graded twists of $C$)
such that $\hdet \mu_B = \hdet \mu_A$ and $\mu_B = \xi_c$ for some $c$.  Note that $B$ is still AS Gorenstein, since 
polynomial extensions preserve this property  by Lemma~\ref{xxlem6.1}(a), and it is standard that graded twists do also.

If instead $A$ is noetherian connected graded AS Gorenstein algebra 
with AS index $\bfl=-1$, then first replace $A$ by $A[y]$ which is connected graded AS Gorenstein 
with AS index $\bfl' = \bfl + 1 \neq -1$, and has $\hdet \mu_{A[y]} = \hdet \mu_A$ by Lemma~\ref{xxlem6.1}(b), since obviously $\mu_{k[y]} = 1$.  Then proceed as in the previous paragraph.

In all cases, we see that $B$ exists as stated, where $B$ is a series of (multi)-graded twists of a polynomial extension of $A$ in one or two variables. 
\end{proof}

Now we are ready to prove homological identity \eqref{HI3}.   The main idea is that for Koszul AS regular algebras $A$, the Nakayama 
automorphism of $A$ is related in a known way to the Nakayama automorphism of its Ext algebra $E$, which is Frobenius, and so 
Lemma~\ref{xxlem3.11} applies.
\begin{theorem}
\label{xxthm6.2} Let $A$ be a noetherian connected graded Koszul AS regular 
algebra. Then $\hdet \mu_A=1$.
\end{theorem}
\begin{proof}
Let $d$ be the global dimension of $A$.  Let $E = E(A) = \bigoplus_{i \geq 0} \Ext_A^i(k, k)$ be 
the Ext algebra of $A$.  Since $A$ is Koszul, we may write$A = T(V)/(R)$ for a 
space of relations $R \subseteq V^{\otimes 2}$, and $E$ is isomorphic to the Koszul dual $A^{!}$ of $A$, 
$T(V^*)/(R^{\perp})$, which is a finite-dimensional algebra generated in degree $1$.  
Moreover, since $A$ is AS regular, it is known that the algebra $E$ is a Frobenius $k$-algebra~\cite{Sm}; 
Let $\nu$ denote its (classical) Nakayama automorphism.
Then by \cite[Theorem 9.2]{VdB1} and Lemma~\ref{xxlem3.5}, the AS index of $A$ is $\bfl = d$ and one has 
\[
\mu_A \vert_V = \xi_{-1}^{d+1} \circ (\nu \vert_{V^*})^*.
\]

By Lemma~\ref{xxlem6.2}, by extending the base field if necessary and replacing $A$ with a multigraded twist of a polynomial extension of $A$, 
we may reduce to the case that $\mu_A = \xi_c$ for some $c$.   (Note that all of these changes preserve the 
AS regular Koszul hypothesis.)  Then by the formula above,  since $\mu_A \vert_V$ is multiplication 
by $c$, we get that $\nu \vert_{V^*}$ is multiplication by $(-1)^{d+1} c$.  Since $E$ is generated in degree $1$, 
$\nu$ acts on the $1$-dimensional degree $d = \bfl$ piece by the scalar $((-1)^{d+1} c)^d = (-1)^{d(d+1)} c^{d} = c^{d}$.  By Lemma~\ref{xxlem3.11}, $\hdet \nu = c^{d} = 1$.  On the other hand, 
we also have that $\hdet \mu_A = c^{\bfl} = c^d$ by Lemma~\ref{xxlem5.3}, so $\hdet \mu_A = 1$.
\end{proof}

As already alluded to in the introduction, we conjecture that the following more general result holds.
\begin{conjecture}
\label{xxconj6.5}
Let $A$ be a noetherian connected graded AS Gorenstein algebra.  Then $\hdet \mu_A = 1$.
\end{conjecture}

\section{Applications}
\label{xxsec7}

We explore some applications of our results above in this section, including Corollaries~\ref{xxcor0.6} and \ref{xxcor0.7}.   We concentrate 
here on connected ${\mathbb Z}^w$-graded AS Gorenstein algebras $A$.   Some of these results depend on knowing 
that $\hdet \mu_A = 1$, as in Conjecture~\ref{xxconj6.5}.

We start with some applications to the calculation of the homological determinant.   
Note that Corollary \ref{xxcor0.5} is a special case of part (b) of the next result.
\begin{lemma}
\label{xxlem7.1}  Suppose that Conjecture~\ref{xxconj6.5} holds.  Let $A$ be a noetherian connected graded AS
Gorenstein algebra.   
\begin{enumerate}
\item  Suppose that $z$ is a homogeneous $\mu_A$-normal nonzerodivisor of positive degree in $A$, so that $\mu_A(z)= cz$ for some $c\in k^\times$, and let $\tau \in \Autz(A)$ be the automorphism such that $za =\tau(a)z$ for all $a\in A$.    Then
$$\hdet \tau=c.$$ 
In particular, if $A$ has trivial Nakayama automorphism, then $\hdet \tau=1$.
\item
Let $\varphi\in \Autz(A)$. Then
$$\hdet \varphi=\mu_{A[t;\varphi]}(t)\; t^{-1}.$$
\end{enumerate}
\end{lemma}

\begin{proof} (a) First note that $\tau(z)=z$.
By Lemma \ref{xxlem1.6}, $\mu_{A/(z)}=
(\mu_A\circ \tau)\mid_{A/(z)}$.   Since $A/(z)$ is also AS Gorenstein, by assumption we have that  
$\hdet_{A/(z)} \mu_{A/(z)}= \hdet_{A/(z)}
(\mu_A\circ \tau)=1$. By \eqref{E3.9.1},
$$\hdet_A (\mu_A\circ \tau)= c \hdet_{A/(z)}
(\mu_A\circ \tau)\mid_{A/(z)}=c$$ as $(\mu_A\circ\tau)(z)=\mu_A(z)=c
z$.  Since $\hdet_A \mu_A=1$, we have 
$$\hdet_A \tau=\hdet_A (\mu_A\circ \tau)=c.$$
The last sentence is a special case.

(b) Let $B=A[t;\varphi]$.  Note that $B$ is $\mb{Z}^2$-graded, with degree $(i,j)$-piece $A_i t^j$.  Thus $\mu_B$ is a ${\mathbb Z}^2$-graded 
automorphism, so $\mu_B(t)=c t$ for some $c\in k^\times$. Let $\tau\in \Autz(B)$
be defined by $\tau(a t^n)=\varphi(a) t^n$ for all $a\in A$ and $n\geq 0$.
Thus $t b =\tau(b) t$ for all $b\in B$. By \eqref{E3.9.1},
$$\hdet_{B} \tau=1 \hdet_A \varphi=\hdet_A \varphi$$
since $\tau(t)=1 t$. Now applying part (a), one sees that
$$\hdet_A \varphi = \hdet_B \tau=c=\mu_B(t)t^{-1},$$
as claimed.
\end{proof}

One of our goals for future work is to explore how one might define homological determinant 
for automorphisms of not necessarily connected graded algebras.  This suggests the 
following question.
\begin{question}
\label{xxque7.5} Do Theorems \ref{xxthm0.2} and 
\ref{xxthm0.3} suggest a way to define the homological determinant $\hdet$ in a
more general setting? For example, $\mu_{A\#H}(\mu_A\# \mu_H)^{-1}$
should be $1\# \Xi^l_{\hdet}$. If $A$ is local (not graded), this
could be a way of defining $\hdet$.

Another possibility is to use the formula in Proposition \ref{xxpro7.2}(c),
$$\hdet \varphi=\mu_{A[t;\varphi]}(t)t^{-1}.$$
In the ungraded case, one can show that
$\mu_{A[t;\varphi]}(t)t^{-1}\in A^\times$ for any $\varphi\in
\Aut(A)$. 
\end{question}

We next study several constructions which help to produce algebras with a trivial Nakayama automorphism, in particular, 
CY algebras.
\begin{proposition}
\label{xxpro7.3}
Let $A$ be a noetherian connected graded AS regular algebra with $\hdet \mu_A = 1$.  Then the 
algebras $B = A[t;\mu_A]$ and $B' = A[t^{\pm 1};\mu_A]$ are CY algebras.
\end{proposition}
\begin{proof}  Since $A$ is AS regular, it is well-known that the Ore extension $B = A[t;\mu_A]$ is also.  So $A[t;\mu_A]$ is skew CY by Lemma~\ref{xxlem1.2}.
Let $C=A[t]$, which as in the proof of Lemma~\ref{xxlem7.1}(b) is $\mb{Z}^2$ graded.
Let $\sigma=(Id_C, \mu_C^{-1})$.  Then the AS index of $C$ is $\bfl(C)=(\bfl(A),1)$ by Lemma~\ref{xxlem6.1}(a), and $\hdet \mu_C = \hdet \mu_A = 1$, by Lemma~\ref{xxlem6.1}(b).  Now by Theorem \ref{xxthm5.4}(a),
$$\mu_{C^{\tilde{\sigma}}}=\mu_C\circ \sigma^{\bfl}\circ \xi^{-1}_{\hdet(\sigma)}=Id_C.$$ Thus  $C^{\tilde{\sigma}}$ is CY. 
Finally, one may check that the function $\psi: C^{\tilde{\sigma}} \to A[t;\mu_A]$ given by the formula  
$\sum_i a_i t^i \mapsto \sum_i \mu^i(a_i) t^i$ is an isomorphism of rings. 

(b)  The ring $B'$ is the localization of $B$ at the set of powers of the normal nonzerodivisor $t$.  It is standard that  
such a localization of a CY algebra remains CY.
\end{proof}

\begin{proof}[Proof of Corollary \ref{xxcor0.6}]
(a)  This is Proposition~\ref{xxpro7.3}(a).

(b) If $\mu_A$ has infinite order, the assertion is equivalent to
Proposition \ref{xxpro7.3}(b).  If $\mu_A$ has finite order, then 
the assertion follows from Corollary~\ref{xxex4.2} by taking $n=1$.

(c) Since $\mu = \mu_A$ has finite order and $\hdet \mu=1$, the invariant subring $A^G$ is
AS Gorenstein by \cite[Theorem 3.3]{JoZ}, where $G$ is the finite
group $\langle \mu \rangle$. Let $C=A^G$ and let $\int=\frac{1}{|G|}
\sum_{g\in G}g$. By
\cite[(E3.1.1), Lemmas 3.2 and 3.5(d)]{KKZ},
$$R^d\Gamma_{\fm_C}(C)^*=R^d\Gamma_{\fm}(A)^* \cdot {\textstyle{ \int}}
=(R^d\Gamma_{\fm}(A)^*)^{G} =(^\mu A^1(-{\bfl}))^{G}$$ as
$C$-bimodules, where $\bfl=\bfl(A)$. Note that the restriction of
$\mu$ to $C$ is the identity. Hence
$$(^\mu A^1(-{\bfl}))^{G}={^1 C^1}(-{\bfl})$$
and therefore the Nakayama automorphism of $C$ is trivial by Lemma
\ref{xxlem3.5}.
\end{proof}
\noindent We remark that if $A$ is PI, then Corollary \ref{xxcor0.6}(a) is a
special case of a result of Stafford-Van den Bergh \cite[Proposition
3.1]{SVdB}.

The results above strongly suggest the following question.
\begin{question}
\label{xxque7.4}
Do the conclusions of Corollary~\ref{xxcor0.6} hold for ungraded skew CY algebras $A$?  In particular, 
when $A$ is skew CY, is $A\rtimes \langle \mu_A\rangle$ always CY?  
\end{question}

We have already seen in the proof of Lemma~\ref{xxlem6.2} how our formula for the Nakayama automorphism 
of a graded twist may allow one to pass to a twist equivalent algebra with simpler Nakayama automorphism.   
We now offer several similar results about how close one might be able to get to a CY algebra through twisting.
\begin{proposition}
\label{xxpro7.2} Let $A$ be a noetherian connected graded AS Gorenstein algebra.  
\begin{enumerate}
\item
Assume $k$ is algebraically closed.
Suppose that $A$ is ${\mathbb Z}^w$-graded with a set of
generators $x_1,\cdots,x_w$ such that $\deg x_i$ is the $i$th unit
vector in ${\mathbb Z}^w$.  If $\bfl(A)\neq 0$, then there is an automorphism
$\sigma\in \Autw(A)$ such that $\mu_{A^{\tilde{\sigma}}} =\xi_c$
for some $c\in k^\times$.
\item
Let $\varphi\in \Autz(A)$ and assume that $\hdet \mu_A = 1$. Then there is a ${\mathbb Z}^2$-graded twist
of $A[t;\varphi]$ that has trivial Nakayama automorphism.
\end{enumerate}
\end{proposition}
\begin{proof}
(a) Since $A$ is ${\mathbb Z}^w$-graded, its Nakayama automorphism
$\mu_A$ is ${\mathbb Z}^w$-graded. Thus $\mu_A(x_i)=a_i x_i$ for
all $i=1,\cdots, w$. Viewing $A$ as a connected graded algebra we
assume that $\bfl:=\bfl(A)\neq 0$. Since $k$ is algebraically closed, there
are $\delta_i\in k^\times$ such that $\delta_i^{-\bfl}=a_i$ for all $i$.
Then $\sigma^{-\bfl}=\mu_A$ where $\sigma=\xi_{\delta}\in \Autw(A)
\subset \Autz(A)$. The assertion follows from Theorem \ref{xxthm5.4}(a)
(with $w=1$).

(b) By Proposition~\ref{xxpro7.3}, $A[t; \mu_A]$ is CY.  Since $\mu_A$ is in the center of $\Autz(A)$ by Theorem~\ref{xxthm3.10}, 
a similar argument as in the proof of Proposition~\ref{xxpro7.3} shows that $A[t; \mu_A]$ is a 
$\mb{Z}^2$-graded twist of $A[t; \varphi]$. 
\end{proof}

For the rest of this section we prove Corollary \ref{xxcor0.7}.
From now on we assume that $k$ is algebraically closed of
characteristic $0$. We use \cite{Bor} as a reference for algebraic
groups. The following two lemmas are presumably known, but we sketch the proofs since 
we lack a reference.
\begin{lemma}
\label{xxlem7.6}
Let $A$ be a finitely generated ${\mathbb Z}^w$-graded algebra, which is locally finite and connected $\mb{N}$-graded with 
respect to the $\mid\mid \;\;\mid\mid$-grading.
Then $\Autw(A)$ is an affine algebraic group over $k$.
\end{lemma}

\begin{proof}  Let A be generated as an algebra by homogeneous elements $\{ a_i \}_{i = 1}^s$, where 
$a_i \in A_{\alpha_i}$.  Choose a graded presentation $F/I \cong A$, where F is a $\mathbb{Z}^w$-graded free
algebra with homogeneous generators $x_i$ which map to the $a_i$.
There is a natural injection $i : \Autw(A) \to \prod_{i=1}^s \GL(A_{\alpha_i})$, where 
$\tau \in \Autw(A)$ corresponds to the product of the bijections it induces of the graded subspaces containing the generating
set, since any $\tau \in \Autw(A)$ is determined by the elements $\tau(a_i)$. Conversely,
choosing arbitrary $b_i \in A_{\alpha_i}$, there is an automorphism $\tau$ in $\Autw(A)$  with
$\tau(a_i) = b_i$  if and only if for every relation $r = \sum d_{j_1, \dots, j_k} x_{j_1} \dots x_{j_k} \in I$, 
we have $\sum d_{j_1, \dots, j_k} b_{j_1} \dots b_{j_k} = 0$ in $A$.   It is easy to see that for any relation $r$ this gives a
closed condition on the choice of $b_i$, and intersecting over all relations (or a generating
set of the ideal of relations) we get that $i$ is a closed regular embedding.
The map $i$ is clearly a group homomorphism, so $\Autw(A)$ is an affine algebraic
group.
\end{proof}

\begin{lemma}
\label{xxlem7.7} Let $d \neq 0$.  Let $G$ be an abelian affine
algebraic group, and let $G^0$ be the connected component of the identity. 
\begin{enumerate}
\item
If $G$ is connected, then for any $x \in G$, there exists $y \in G$ such
that $y^d = x$.
\item For any $x \in G$, there exists $z \in G$
of finite order and $y \in G^0$ such that $xz \in G^0$ and $xz = y^d$.
\end{enumerate}
\end{lemma}

\begin{proof}
(a) Since $G$ is abelian, the map $\phi: G \to G$ defined by $\phi(y) = y^d$
is a homomorphism of algebraic groups. The image $H$ is a closed
algebraic subgroup of $G$. Hence the quotient algebraic group
$P = G/H$ exists, and this is again affine \cite[Theorem 6.8]{Bor}.
It is also still abelian and connected since $G$ is. Note that $z^d = e$
for all $z \in P$. We claim that $P$ is the trivial group $\{e\}$. Since the elements of 
$P$ are all $d$th roots of $e$, it follows  that 
$P$ consists of semisimple elements under any embedding of $P$ in a matrix group.  Then 
$P$ can be embedded in a torus $(k^\times)^n$ for some $n$, by \cite[Proposition 8.4]{Bor}.  Since $P$ is
connected and a torus contains only finitely many elements of order $d$, we must have that 
$P$ is a point.  In other words, $\phi$ is surjective and the result follows.

(b) We just have to produce an element $z$ of finite order such that
$xz \in G^0$.  The existence of $y$ follows from part (a).
Let $H$ be the connected component containing $x$.  Choose any $w \in H$
and $m \geq 1$ such that $w^m \in G^0$.  By part (a), there
is $v \in G^0$ such that $v^m = w^m$.  Then $z = vw^{-1}$ satisfies
$z^m = e$, so $z$ has finite order.  Moreover, clearly $xz \in G^0$,
since $x \in H$, $w \in H$, and $v \in G^0$.
\end{proof}

\begin{theorem}
\label{xxthm7.8}
Let $A$ be a noetherian connected graded AS Gorenstein algebra with $\bfl(A)\neq 0$ and 
with $\hdet \mu_A = 1$.
Then there is a $\sigma\in \Autz(A)$ such that $\mu_{A^{\tilde{\sigma}}}$
is of finite order.
\end{theorem}

\begin{proof}
By Theorem \ref{xxthm3.10}, $\mu_A$ is in the center $Z$ of $\Autz(A)$.
Note that $\Autz(A)$ is an affine algebraic group by Lemma~\ref{xxlem7.6} 
and that $Z$ is a closed algebraic subgroup of $\Autz(A)$.  It suffices to work in $Z$ 
for the rest of the proof.

Let $x=\mu_A^{-1}$. By Lemma \ref{xxlem7.7}(b), there is an $z \in Z$
of finite order and $y\in Z^0$, such that $xz \in Z^0$ and $y^{\bfl}= x z$
where $\bfl=\bfl(A)$. Let $\sigma=y$. Then by Theorem \ref{xxthm5.4}(a),
$$\mu_{A^{\tilde{\sigma}}}=\mu_A\circ \sigma^{\bfl}\circ
\xi_{\hdet(\sigma)}^{-1} = z\circ \xi_{\hdet(\sigma)}^{-1}.$$
Since $y^{\bfl}=xz$, we have
$$\hdet(\sigma)^{\bfl}=\hdet( y^{\bfl})=\hdet x\hdet z=\hdet z$$
since $\hdet x=(\hdet \mu_A)^{-1}=1$ by assumption.
Since $z$ has finite order, $\hdet z$ has finite
order.  Thus $\hdet(\sigma)$ has finite order.  Finally, we see that 
$\mu_{A^{\tilde{\sigma}}}$ has finite order.
\end{proof}

We conclude the paper with the proof of the last corollary from the introduction.
\begin{proof}[Proof of Corollary \ref{xxcor0.7}]
Since $A$ is AS regular, $\bfl(A)>0$ \cite[Proposition 3.1]{SteZh}. The assertion follows from
Theorem \ref{xxthm7.8}.
\end{proof}

\providecommand{\bysame}{\leavevmode\hbox to3em{\hrulefill}\thinspace}
\providecommand{\MR}{\relax\ifhmode\unskip\space\fi MR }
\providecommand{\MRhref}[2]{%
  \href{http://www.ams.org/mathscinet-getitem?mr=#1}{#2} }
\providecommand{\href}[2]{#2}


\begin{thebibliography}{10}

\bibitem[ASc]{ASc}
M. Artin and W.F. Schelter,
\emph{Graded algebras of global dimension 3},
Adv. in Math. \textbf{66} (1987), no. 2, 171--216.




\bibitem[AZ]{AZ}
M. Artin and J.J. Zhang,
\emph{Noncommutative projective schemes},
Adv. Math. \textbf{109} (1994), no. 2, 228--287.

\bibitem[BM]{BM}
R. Berger and N. Marconnet,
\emph{Koszul and Gorenstein properties for homogeneous algebras},
Algebr. Represent. Theory {\bf 9} (2006), no. 1, 67--97.


\bibitem[Boc]{Boc}
R. Bocklandt,
\emph{Graded Calabi Yau algebras of dimension 3},
J. Pure Appl. Algebra {\bf 212} (2008), no. 1, 14--32.

\bibitem[BSW]{BSW}
R. Bocklandt, T. Schedler and M. Wemyss,
\emph{Superpotentials and higher order derivations},
J. Pure Appl. Algebra {\bf 214} (2010), no. 9, 1501--1522,

\bibitem[Bor]{Bor}
A. Borel, \emph{Linear algebraic groups}, second ed., Springer-Verlag, New
  York, 1991. 



\bibitem[Br]{Br}
K.A. Brown,
\emph{Noetherian Hopf algebras},
Turkish J. Math. {\bf 31} (2007), suppl., 7--23.

\bibitem[BZ]{BZ}
K.A. Brown and J.J. Zhang,
\emph{Dualising complexes and twisted Hochschild (co)homology
for Noetherian Hopf algebras},
J. Algebra {\bf 320} (2008), no. 5, 1814--1850.

\bibitem[CWZ]{CWZ}
K. Chan, C. Walton and J.J. Zhang,
\emph{Hopf actions and Nakayama automorphisms}, preprint (2012), arXiv:1210.6432v1.

\bibitem[Fa]{Fa}
M. Farinati,
\emph{Hochschild duality, localization, and smash products},
J. Algebra {\bf 284} (2005), no. 1, 415--434.

\bibitem[DNR]{DNR}
S. D{\u{a}}sc{\u{a}}lescu, C. N{\u{a}}st{\u{a}}sescu, and \c{S}. Raianu,
``Hopf Algebras: An Introduction'', Monographs and Textbooks in Pure and
Applied Mathematics, vol. 235, Marcel Dekker, Inc., New York, 2001.

\bibitem[DV]{DV}
M. Dubois-Violette,
\emph{Multilinear forms and graded algebras},
J. Algebra {\bf 317} (2007), no. 1, 198--225.

\bibitem[FMS]{FMS}
D. Fischman, S. Montgomery and H.-J. Schneider,
Frobenius extensions of subalgebras of Hopf algebras,
Trans. Amer. Math. Soc. {\bf 349} (1997), no. 12, 4857--4895.


\bibitem[FV]{FV}
Gunnar Fl{\o}ystad and Jon~Eivind Vatne, \emph{Artin-{S}chelter regular
  algebras of dimension five}, Algebra, geometry and mathematical physics,
  Banach Center Publ., vol.~93, Polish Acad. Sci. Inst. Math., Warsaw, 2011,
  pp.~19--39.

\bibitem[Gi]{Gi}
V. Ginzburg,
\emph{Calabi-Yau algebras}, arXiv:math/0612139 v3, 2006.


\bibitem[IR]{IR}
O. Iyama and I. Reiten,
\emph{Fomin-Zelevinsky mutation and tilting modules over
Calabi-Yau algebras},
Amer. J. Math. {\bf 130} (2008), no. 4, 1087--1149.

\bibitem[JiZ]{JiZ}
N. Jing and J.J. Zhang,
\emph{Gorensteinness of invariant subrings of quantum algebras},
J. Algebra \textbf{221} (1999), no. 2, 669--691.


\bibitem[JoZ]{JoZ}
P. J{\/o}rgensen and J.J. Zhang,
\emph{Gourmet's guide to Gorensteinness},
Adv. Math. {\bf 151} (2000), no. 2, 313--345.

\bibitem[Kay]{Kay}
A. Kaygun, \emph{Hopf-{H}ochschild (co)homology of module algebras},
  Homology, Homotopy Appl. \textbf{9} (2007), no.~2, 451--472. 


\bibitem[Ke]{Ke}
B. Keller,
\emph{Calabi-Yau triangulated categories},
Trends in representation theory of algebras and related topics, 467-489,
EMS Ser. Congr. Rep., {\it Eur. Math. Soc., Z{\" u}rich}, 2008.

\bibitem[KK]{KK}
E. Kirkman and J. Kuzmanovich,
\emph{Fixed subrings of Noetherian graded regular rings},
J. Algebra {\bf 288} (2005), no. 2, 463--484.


\bibitem[KKZ]{KKZ}
E. Kirkman, J. Kuzmanovich and J.J. Zhang,
\emph{Gorenstein subrings of invariants under Hopf algebra actions},
J. Algebra {\bf 322} (2009), no. 10, 3640--3669.

\bibitem[Ko]{Ko}
M. Kontsevich,
\emph{Triangulated categories and geometry}, Course at the {\' E}cole
Normale Sup{\' e}rieure, Paris, Notes taken by J. Bella{\" i}che, J.-F.
Dat, I. Marin, G. Racinet and H. Randriambololona, 1998.


\bibitem[La]{La}
T.\,Y.~Lam,
\emph{Lectures on Modules and Rings},
Graduate Texts in Mathematics, 189, Springer-Verlag, New York, 1999.

\bibitem[LM]{LM}
P. Le Meur,
\emph{Crossed productd of Calabi-Yau algebras by finite groups},
preprint, (2010) arXiv:1006.1082.

\bibitem[LiWZ]{LiWZ}
L.-Y. Liu, Q.-S. Wu and C. Zhu,
\emph{Hopf Action on Calabi-Yau algebras},
New trends in noncommutative algebra, 189--209, Contemp. Math., 562,
Amer. Math. Soc., Providence, RI, 2012.


\bibitem[Lo]{Lo}
Martin Lorenz, \emph{Representations of finite-dimensional {H}opf algebras}, J.
  Algebra \textbf{188} (1997), no.~2, 476--505. 

\bibitem[LuWZ1]{LuWZ1}
D.-M. Lu, Q.-S. Wu and J.J. Zhang,
\emph{Hopf algebras with rigid dualizing complexes},
Israel J. Math. {\bf 169} (2009), 89--108.

\bibitem[LuWZ2]{LuWZ2}
D.-M. Lu, Q.-S. Wu and J.J. Zhang,
\emph{Homological integral of Hopf algebras},
Trans. Amer. Math. Soc. {\bf 359} (2007), no. 10, 4945--4975.


\bibitem[Lu]{Lu}
V. A. Lunts,
\emph{Categorical resolution of singularities},
J. Algebra \textbf{323} (2010), no. 10, 2977--3003.

\bibitem[M-VS]{M-VS}
R. Martinez-Villa and {\/O}. Solberg,
\emph{Artin-Schelter regular algebras and categories},
J. Pure Appl. Algebra {\bf 215} (2011), no. 4, 546--565.



\bibitem[MM]{MM}
H. Minamoto and I. Mori,
\emph{The structure of AS-Gorenstein algebras},
Adv. Math. \textbf{226} (2011), no. 5, 4061--4095.

\bibitem[Mo]{Mo}
S. Montgomery,
``Hopf Algebras and Their Actions on Rings'',
CBMS Reg. Conf. Ser. Math., vol. 82, Amer. Math. Soc.,
Providence, RI, 1993.

\bibitem[Mu]{Mu}
W. Murray,
\emph{Bilinear forms on Frobenius algebras},
J. Algebra {\bf 293} (2005), no. 1, 89--101.

\bibitem[RRZ]{RRZ}
M. Reyes, D. Rogalski and J.J. Zhang,
\emph{Homological determinant for non-connected skew Calabi-Yau algebras},
in preparation.

\bibitem[Sm]{Sm}
S. Paul Smith,
\emph{Some finite-dimensional algebras related to elliptic curves},
in ``Representation theory of algebras and related topics
(Mexico City, 1994),'' 315--348, CMS Conf. Proc. 19, Amer. Math. Soc.,
Providence, RI, 1996.

\bibitem[SVdB]{SVdB}
J.T. Stafford and M. Van den Bergh,
\emph{Noncommutative resolutions and rational singularities},
Special volume in honor of Melvin Hochster,
Michigan Math. J. {\bf 57} (2008), 659--674.


\bibitem[SteZh]{SteZh}
Darin~R. Stephenson and James~J. Zhang, \emph{Growth of graded {N}oetherian
  rings}, Proc. Amer. Math. Soc. \textbf{125} (1997), no.~6, 1593--1605.
 
\bibitem[VdB1]{VdB1}
M. Van den Bergh,
\emph{Existence theorems for dualizing complexes over
non-commutative graded and filtered rings},
J. Algebra \textbf{195} (1997), no. 2, 662--679.

\bibitem[VdB2]{VdB2}
M. Van den Bergh,
\emph{A relation between Hochschild homology and cohomology for
Gorenstein rings} Proc. Amer. Math. Soc. 126 (1998), no. 5, 1345--1348;
\emph{Erratum to: ``A relation between Hochschild homology and
cohomology for Gorenstein rings''}, Proc. Amer. Math. Soc. 130 (2002),
no. 9, 2809-2810.

\bibitem[VdB3]{VdB3}
M. Van den Bergh, \emph{Calabi-Yau algebras and superpotentials},
preprint, (2010), arXiv:1008.0599.


\bibitem[WZ]{WZ}
Q.-S. Wu and J.~J. Zhang, \emph{Dualizing complexes over noncommutative local
  rings}, J. Algebra \textbf{239} (2001), no.~2, 513--548.


\bibitem[WZhu]{WZhu}
Q.-S. Wu and C. Zhu,
\emph{Skew group algebras of Calabi-Yau algebras}, J. Algebra
{\bf 340} (2011), 53--76.


\bibitem[Ye]{Ye}
A. Yekutieli,
\emph{The rigid dualizing complex of a universal enveloping algebra},
J. Pure Appl. Algebra \textbf{150} (2000), no. 1, 85--93.


\bibitem[YeZ1]{YeZ1}
A. Yekutieli and J.J. Zhang,
\emph{Rings with Auslander dualizing complexes},
J. Algebra \textbf{213} (1999), no. 1, 1--51.

\bibitem[YeZ2]{YeZ2}
A. Yekutieli and J.J. Zhang,
\emph{Homological transcendence degree},
Proc. London Math. Soc. (3) {\bf 93} (2006), no. 1, 105--137.

\bibitem[YuZ1]{YuZ1}
X.L. Yu and Y.H. Zhang,
\emph{Calabi-Yau pointed Hopf algebras of finite Cartan type,}
preprint (2011), arXiv:1111.3342v2.

\bibitem[YuZ2]{YuZ2}
X.L. Yu and Y.H. Zhang,
\emph{The Calabi-Yau property of Hopf algebras and braided Hopf algebras},
preprint (2011), arXiv:1111.4201.


\bibitem[Zh]{Zh}
J.J. Zhang,
\emph{Twisted graded algebras and equivalences of graded categories,}
Proc. London Math. Soc. (3) {\bf 72} (1996), no. 2, 281--311.


\end{thebibliography}
\end{document}